\documentclass[openany, 12pt]{article}

\usepackage[latin1]{inputenc}
\usepackage[T1]{fontenc}
\usepackage{verbatim}
\usepackage[main=english]{babel} 
\usepackage{mathpazo}
\usepackage{color}
\usepackage{geometry}
\usepackage{enumitem}

\usepackage{amsfonts,amssymb,amsthm,amsmath}
\usepackage{makecell}
\usepackage{capt-of}
\usepackage{multirow}
\usepackage{tabularx}
\usepackage[justification=centering]{caption}

\theoremstyle{plain}
\newtheorem{thm}{Theorem}
\newtheorem*{thm*}{Theorem}
\newtheorem{cor}{Corollary}
\newtheorem{prop}{Proposition}
\newtheorem{lem}{Lemma}
\newtheorem*{lem*}{Lemma}
\newtheorem{defi}{Definition}
\newtheorem*{defi*}{Definition}

\usepackage{breakcites}

\usepackage{tikz}
\usetikzlibrary{shapes.geometric}

\usepackage{setspace}
\usepackage{graphicx,graphics}
\usepackage[space]{grffile}

\usepackage{color}
\usepackage{array}
\usepackage{xcolor}

\usepackage{here}

\allowdisplaybreaks

\usepackage[english]{hyperref}

\usepackage{xifthen}
\usepackage{xargs}

\usepackage{mathtools}

\usepackage{caption}
 
\newcommand\numberthis{\addtocounter{equation}{1}\tag{\theequation}}

\usepackage{eurosym}

\usepackage{longtable}
\usepackage{tabularx}

\usepackage{lmodern}

\usepackage{listings}

\usepackage{calrsfs}

\usepackage{dsfont}

\usepackage{changepage}

\usepackage{subfiles}

\usepackage{comment}

\usepackage{array}
\usepackage{makecell}

\usepackage{hyperref}

\hypersetup{
    colorlinks,
    citecolor=black,
    filecolor=red,
    linkcolor=blue,
    urlcolor=blue
}
\usepackage{tcolorbox}

\title{Benign overfitting and adaptive nonparametric regression}
\usepackage{authblk}
\author[1]{Julien Chhor}
\author[1]{Suzanne Sigalla}
\author[1]{Alexandre B. Tsybakov}
\affil[1]{CREST-ENSAE}
\date{}

\def\matR{\mathbf{R}}
\def\matE{\mathbf{E}}

\def\matN{\mathbf{N}}
\def\matP{\mathbf{P}}
\def\lmin{\lambda_{\min}}

\newcommand{\norm}[1]{\left\|#1\right\|_{*}}
\newcommand{\Bd}{\mathcal{B}_d}

\newcommand{\betamax}{\beta_{\max}}

\newcommand{\newpar}{\par\medskip}

\begin{document}
\maketitle

\vspace{-15mm}

\begin{center}
    Contact: julien.chhor@ensae.fr, suzanne.sigalla@ensae.fr, alexandre.tsybakov@ensae.fr
\end{center}

\vspace{1mm}

\begin{abstract}
    In the nonparametric regression setting, we construct an estimator which is a continuous function interpolating the data points with high probability, while attaining minimax optimal rates under mean squared risk on the scale of Hölder classes adaptively to the unknown smoothness.
\end{abstract}

\vspace{3mm}


\section{Introduction}

Benign overfitting has attracted a great deal of attention in the recent years. It was initially motivated by the fact that deep neural networks have good predictive properties even when perfectly interpolating the training data~\cite{belkin2019reconciling}, \cite{belkin2018understand}, \cite{zhang2021understanding}, \cite{belkin2021fit}. 
Such a behavior stands in strong contrast with the classical point of view that perfectly fitting the data points is not compatible with predicting well. With the aim of understanding this new phenomenon, a series of recent papers studied benign overfitting in linear regression setting, see ~\cite{bartlett2020benign}, \cite{tsigler2020benign}, \cite{chinot2020robustness},  \cite{muthukumar2020harmless}, \cite{bartlett2021failures}, \cite{lecue2022} and the references therein. 
The main conclusion for the linear model is that an unbalanced spectrum of the design matrix and over-parametrization, which in a sense approaches the model to non-parametric setting, are essential for benign overfitting to occur in linear regression. 
Extensions to kernel ridgeless regression were considered  in \cite{liang2020just} when the sample size $n$ and the dimension $d$ were assumed to satisfy $n \asymp d$, and in ~\cite{liang2020multiple} for a more general case $d \asymp n^\alpha$ for $\alpha \in (0,1)$. These papers give data-dependent upper bounds on the risk that can be small assuming favorable spectral properties of the data and the kernel matrix. On the other hand, if $d$ is constant (independent of $n$) then the least-norm interpolating estimator with respect to the Laplace kernel is inconsistent~\cite{rakhlin2019consistency}. 

In the line of work cited above, benign overfitting was understood as achieving simultaneously interpolation and prediction consistency, or possibly, consistency with some suboptimal rates. On the other hand, it was shown that, in non-parametric regression setting,  interpolating estimators can attain minimax optimal rates \cite{belkin2018does}. Namely, it is proved in \cite{belkin2018does} that interpolation with minimax optimal rates can be achieved by Nadaraya-Watson estimator with a singular kernel. 

The idea of using singular kernels can be traced back to  \cite{shepard1968} giving start to popular techniques in image processing referred to as Shepard interpolation. In statistical language, Shepard interpolant is nothing else but the Nadaraya-Watson estimator with kernel $K(u)=1/\Vert u\Vert^2$, where $\Vert \cdot\Vert$ denotes the Euclidean norm and $u\in\matR^2$. Unaware of Shepard's work and its subsequent extensive use in
image processing, \cite{devroye1998hilbert} considered the same 
estimator in general dimension $d$, that is, with the kernel $K(u)=\Vert u\Vert^{-d}$ for $u\in\matR^d$, and proved that the Nadaraya-Watson estimator with such a kernel is consistent in probability but fails to be pointwise almost surely consistent. However, this kernel is not integrable and has a peculiar property  that the bandwidth cancels out from the definition of the estimator. Thus, the bias cannot be controlled and the bias-variance trade-off argument based on bandwidth selection does not apply. It remains unclear whether some rates of convergence can be achieved by such an estimator. Therefore, it was suggested in \cite{belkin2018does,belkin2018does-bis} to localize and modify the kernel as $K (u) =\Vert u\Vert^{-a} \mathbf{1}(\Vert u\Vert\le 1)$ where 
$0 < a < d/2$ rather than $a = d$ and $\mathbf{1}(\cdot)$ denotes the indicator function. The estimator with such a weaker type of singularity is also interpolating, and it was shown in \cite{belkin2018does,belkin2018does-bis} that it achieves the minimax rates of convergence on the $\beta$-Hölder classes with $0<\beta\le 2$. Also, \cite{belkin2018overfitting} proved a similar claim for the $k$ nearest neighbor analog of this 
estimator with $0<\beta\le 1$. However, those results were restricted to functions with low smoothness $\beta$ and the suggested estimators were not adaptive to $\beta$. 

In this paper, we show that:
\vspace{-2mm}
\begin{itemize}
    \item[(i)] interpolating estimators attaining minimax optimal rates on $\beta$-Hölder classes can be obtained for any smoothness $\beta>0$,
    \vspace{-2mm}
    \item[(ii)] estimators with such properties can be constructed adaptively to the unknown smoothness $\beta \in (0,\betamax]$, for any $\betamax>0$, and to the unknown parameter $L>0$ of the Hölder class of regression functions.
\end{itemize}
\vspace{-2mm}
The estimators that we consider to achieve (i) are local polynomial estimators (LPE) with singular kernels.
In order to obtain adaptive estimators achieving (ii), we  apply aggregation techniques to a family of LPE with singular kernels. 

As a by-product, we obtain non-asymptotic bounds for the squared risk of LPE in classical setting with non-singular kernels. To the best of our knowledge, such bounds are missing in the existing literature on LPE that was mainly focused on asymptotic properties such as convergence in probability or pointwise asymptotic normality, cf. \cite{stone1980,stone1982optimal,tsybakov1986,fan-gijbels}.

Note that local polynomial method with singular kernels has been used as interpolation tool in numerical analysis,  starting from \cite{lancaster}. It was also invoked in the context of non-parametric regression in \cite{katkovnik_book}. However, \cite{lancaster,katkovnik_book} only discussed functional properties, such as the smoothness of interpolants, rather than their statistical behavior.

\section{Preliminaries}

\subsection{Notation}

For any vector $x=(x_1,\dots,x_d)\in \matR^d$ and any  multi-index $s=(s_1,\dots,s_d) \in \matN^d$, we define
\vspace{-7mm}

\begin{align*}
\begin{array}{lll}
    \rvert s \rvert = \sum\limits_{i=1}^ds_i, &&s!=s_1!\dots s_d!\\
    &&\\
     x^s=x_1^{s_1}\dots x_d^{s_d} && D^s= \frac{\partial^{s_1+\dots+s_d}}{\partial x_1^{s_1}\dots \partial x_d^{s_d}}.
\end{array}
\end{align*}
We denote by $\|\cdot\|$ the Euclidean norm, and by ${\rm Card}(J)$ the cardinality of set $J$. For any integer $k\in \matN^*$, we set $[k] = \{1,\dots,k\}$. For any $x \in \matR^d$, $r>0$, we denote by $\mathcal{B}_d(x,r)$ the closed Euclidean ball centered at $x$ with radius $r$. We set for brevity $\Bd = \mathcal{B}_d(0,1)$.
For any $\beta >0$, we denote by $\lfloor \beta \rfloor$ the maximal integer less than $\beta$, and by $\lceil \beta \rceil$ the minimal integer greater than $\beta$. We use symbols $C,C'$ to denote positive constants that can vary from line to line.

For any $k>0$, we denote by $I_k$ the identity matrix of size $k$. For any square matrix $M$, the writing $M \succ 0$ means that $M$ is positive definite. For any matrix $M$, we denote by $M^+$ its Moore-Penrose inverse, and by $\| M \|_{\infty}$ its spectral norm.

\subsection{Model}

Let 
$(X,Y)$ be a pair of random variables in $\matR^d\times \matR$ 
with distribution $P_{XY}$ and assume that we are given $n$ i.i.d. observations $\mathcal{D}: = \left\{(X_1, Y_1),\dots,(X_n,Y_n)\right\}$ with distribution $P_{XY}$.
We denote by $P_X$ the marginal distribution of $X$ and assume that it admits a density $p$ with respect to the Lebesgue measure on the compact set $\text{Supp}(p)$.
We assume that for all $x \in \text{Supp}(p)$, the regression function $f(x) = \matE(Y \rvert X=x)$ exists and is finite. Set $\xi(X)=Y-\matE(Y \rvert X)$. Equivalently, the model can be written as $Y_i = f(X_i) + \xi(X_i)$, where $\matE (\xi(X_i)|X_i)=0$. We make the following assumptions.

\vspace{2mm}

{\bf Assumption $\mathbf{(A1)}$.} {\it $\matE (|\xi(X)|^{2+\delta}|X=x)\le C_{\xi} $ for all $x \in \text{Supp}(p)$, where $\delta$ and $C_{\xi}$ are positive constants. }

\textbf{Assumption $\mathbf{(A2)}$.} {\it  The random vector $X$ is distributed with Lebesgue density $p(\cdot)$ such that $p \in [p_{\min}, p_{\max}]$ where $p_{\max}\geq p_{\min}>0$.
The support $\text{Supp}(p)$ of $p$ is a convex compact set contained in $\Bd$.}

\vspace{2mm}

For any estimator $f_n$ of $f$ based on the sample $\mathcal{D}$,
 we consider the following $L_2$-loss :
\begin{align*}
    \|f_n - f\|^2_{L_2} = \matE_X\left(\left[f_n(X)-f(X)\right]^2\right)= \int \left[f_n(x)-f(x)\right]^2 p(x) dx,
\end{align*}
where $\matE_X$ denotes the expectation with respect to $P_X$.
We define the expected risk as
$
    \matE \left[\|f_n - f\|^2_{L_2}\right],
$
where $\matE$ denotes the expectation with respect to the distribution of $\mathcal{D}$.
\begin{defi}[Interpolating estimator]
An estimator $f_n$ of $f$ based on a sample $\mathcal{D}=\{(X_1,Y_1),\dots,(X_n,Y_n)\}$ is called interpolating over 
$\mathcal{D}$ if 
 $f_n(X_i)=Y_i$ for $i=1,\dots,n$.
\end{defi}

\subsection{Hölder classes of functions}\label{subsec:modif_Holder}

For any $k$-linear form $A:(\matR^d)^k \longrightarrow \matR$, we define its norm as follows
\begin{equation}\label{def:norm}
    \norm{A} := \sup \Big\{\big|A[h_1,\dots,h_k]\big|: \|h_j\|\leq 1, j \in [k]\Big\}.
\end{equation}
Given a $k$-times continuously differentiable function $f:\matR^d \longrightarrow \matR$ and $x \in \matR^d$, we denote by $f^{(k)}(x):(\matR^d)^k \longrightarrow \matR$ the following $k$-linear form
\begin{align*}
    f^{(k)}(x) [h_1,\dots,h_k] = \hspace{-3mm} \sum_{|m_j|=1, \forall j \in [k]} \hspace{-3mm} D^{m_1+\dots + m_k}f(x) h_1^{m_1} \dots h_k^{m_k}, ~~ \forall h_1,\dots, h_k \in \matR^d,
\end{align*}
where $m_1,\dots, m_k \in \matN^d$ are multi-indices.
Throughout the paper, we will consider the following H\"older class of functions. 
\begin{defi}\label{def:Holder}
Let $\beta>0$, $L >0$, and let $f: \Bd \longrightarrow \matR$ be a $\ell=\lfloor \beta \rfloor$ times continuously differentiable function. We denote by $\Sigma(\beta,L)$ the set of all functions $f$ defined on $\Bd$ such that 
\begin{align*}
    \max_{0\le k \leq \ell} \sup_{x\in \Bd}  \norm{f^{(k)}(x)} + \sup_{x,x'\in \Bd} \frac{\norm{f^{(\ell)}(x) -f^{(\ell)}(x')}}{\|x-x'\|^{\beta - \ell}} \leq L.
\end{align*}
\end{defi}

These classes of functions have nice embedding properties that will be needed to prove our result on adaptive estimation. For $\beta' \leq \beta \leq 1$, we clearly have $\Sigma(\beta,L) \subseteq \Sigma(\beta', L)$. Analogous embedding is valid for $\beta>1$ as stated in the next lemma proved in the Appendix.

\begin{lem}\label{lem:Holder_nested}
For any $0 < \beta' \leq \beta$ and $L>0$ we have $\Sigma(\beta,L) \subseteq \Sigma(\beta', 2L)$.
\end{lem}

The class $\Sigma(\beta,L)$ is closely related to several differently defined Hölder classes used in the literature. One of them is based on Taylor approximation, cf., for example, \cite{stone1980}.
For any $x \in \matR^d$ and any $\ell$ times continuously differentiable real-valued function $f$ on $\matR^d$, we denote by $Tf_x$ its Taylor polynomial of degree $\ell$ at point $x$:
\begin{align*}
    Tf_x(x') = \sum_{0\le \rvert s \rvert \leq \ell}\frac{(x-x')^s}{s!}D^s f(x').
\end{align*}

\begin{lem}\label{lem:polynomial}
Let $\beta >0$, $L>0$ and $f \in \Sigma(\beta,L)$. Then for all $x,y \in \Bd$, and $\ell =\lfloor \beta \rfloor$ it holds that
\begin{align*}
    \left|f(x) - Tf_y(x)\right| \leq \frac{L}{\ell !} \|x\!-\!y\|^\beta.
\end{align*}
\end{lem}
Thus, we have $\Sigma(\beta,L) \subseteq \Sigma'(\beta,L/\lfloor \beta \rfloor !)$, where $\Sigma'(\beta,L')$ stands for the class of all functions $f$ satisfying the relation $
    \left|f(x) - Tf_y(x)\right| \leq L' \|x\!-\!y\|^\beta.
$

Next, considering one more definition of Hölder class:
\begin{align*}
    \widetilde \Sigma(\beta,L)  = \left\{f : \Bd \rightarrow \matR : \sup_{x,x'} \frac{\norm{f^{(\ell)}(x) -f^{(\ell)}(x')}}{\|x-x'\|^{\beta - \ell}} \leq L\right\}
\end{align*}
we also immediately have that $\Sigma(\beta,L) \subseteq \widetilde \Sigma(\beta,L)$. It follows from~\cite{stone1982optimal} that the minimax estimation rate on the class $\widetilde \Sigma(\beta,L)$ under the squared loss that we consider below is $n^{-\frac{2\beta}{2\beta+d}}$ up to constants depending only on $\beta$ and $d$. 
Notice that the functions in $\widetilde \Sigma(\beta,L)$  used in the lower bound construction in~\cite{stone1982optimal} can be rescaled into functions in $\Sigma(\beta,L)$ by multiplying by a factor depending only on $\beta$ and $d$.
Hence, the lower bound construction in~\cite{stone1982optimal} remains valid for the class $\Sigma(\beta,L)$. It implies that the minimax rate of estimation on the class $\Sigma(\beta,L)$ is  $n^{-\frac{2\beta}{2\beta+d}}$. 
In conclusion, though $\Sigma(\beta, L)$ is a subclass of  suitable Hölder classes $\Sigma'$ and $\widetilde \Sigma$ it is not substantially smaller, in the sense that estimation over these classes is essentially equally difficult.

\section{Local polynomial estimators and interpolation}


For $\ell \in \matN$ let $C_{\ell,d} = {{\ell + d} \choose {d}}$ be the cardinality of the set of multi-indices $\{ s=(s_1,\dots,s_d)\in\matN^d, 0\le \rvert s \rvert \leq \ell\}$. 
%
We assume that the elements $s^{(1)},\dots,s^{(C_{\ell,d})}$ of this set are ordered according to the increasing values of $|s|$, and in an arbitrary way for equal values of $|s|$. In particular, $s^{(1)}=(0,\dots,0)$.
For any $u \in \matR^d$, define the vector $U(u)\in \matR^{C_{\ell,d}}$ as follows:
\begin{align*}
    U(u) := \left( \frac{u^s}{s!} \right)_{\rvert s \rvert \leq \ell},
\end{align*}  
where the components of $U(u)$ are ordered in the same way as $s^{(i)}$'s. In particular, the first component of $U(u)$ is 1 for any $u$.

The definition of local polynomial estimator usually given in the literature is as follows, cf., e.g., \cite{tsybakov2008introduction}.
Let $K:\matR^d \rightarrow \matR_+$ be a kernel, $h>0$ be a bandwith and $\ell \geq0$ be an integer. Consider a vector $\hat{\theta}_n(x) \in \matR^{C_{\ell,d}}$ such that 
\begin{align}
\label{eq:eq1}
    \hat{\theta}_n(x) \in \underset{\theta \in \matR^{C_{\ell,d}}}{\rm argmin} \ \sum_{i=1}^n \left[Y_i - \theta^{\top} U\left(\frac{X_i-x}{h} \right) \right]^2 K \left(\frac{X_i-x}{h} \right)
\end{align}
Then 
\begin{align}
\label{eq:eq4a}
    f_n(x) = U^\top (0) \hat{\theta}_n(x)
\end{align}
is called a local polynomial estimator of order $\ell$ 
of $f(x)$. Note that $ f_n(x)$ is the first component of $\hat{\theta}_n(x)$.

However, this definition is not convenient for our purposes. First, 
$\hat{\theta}_n(x)$ is not uniquely defined for such $x\in \matR^d$ that the matrix
\begin{align*}
    B_{nx}: = \frac{1}{n{ h^d}} \sum_{i=1}^n U\left( \frac{X_i-x}{h}\right) U^{\top }\left( \frac{X_i-x}{h}\right) K\left( \frac{X_i-x}{h}\right)\in \matR^{C_{\ell,d}\times C_{\ell,d}}
\end{align*}
is degenerate. Furthermore, $\hat{\theta}_n(x)$ is not defined for $x=X_i$ if the kernel $K$ has a singularity at 0, which will be the main case of interest in what follows. Therefore, we adopt the following slightly different definition. 

\begin{defi}[Local polynomial estimator]
\label{def:LocPolEst}
If the kernel $K$ is bounded then the local polynomial estimator of order $\ell$ (or shortly, LP($\ell$) estimator) of $f(x)$ at point $x$ is defined as
\begin{align}
\label{eq:eq4}
    f_n(x)= \sum_{i=1}^n Y_i W_{ni}(x),
\end{align}
where, for $i=1,\dots,n$, the weights $W_{ni}(x)$ are given by
\begin{align}
\label{eq:eq4b}
W_{ni}(x) = \frac{U^{\top}(0)}{n{h^d}} B_{nx}^{+} U\left( \frac{X_i-x}{h}\right) K\left( \frac{X_i-x}{h}\right).
\end{align}
If the kernel $K$ has a singularity at $0$, that is, $\lim_{u \to 0} K(u) = +\infty$, then the LP($\ell$) estimator of $f(x)$ at point $x\notin \{X_1,\dots, X_n\}$ is still defined by \eqref{eq:eq4} 
while we set, for $j=1,\dots,n$,
\begin{align}
\label{eq:eq4c}
f_{n}(X_j) = \limsup\limits_{z \to X_j} f_{n}(z).
\end{align}
\end{defi}
 The purpose of \eqref{eq:eq4c} is to provide a valid definition for kernels with singularity at 0.  
We introduce $\limsup$ in \eqref{eq:eq4c} for formal reasons. In the cases of our interest described in the next lemma there exists an exact limit in \eqref{eq:eq4c}: $\lim_{x\to X_j}f_n(x)=Y_j$ for all $j \in [n]$, which means that the estimator $f_n$ is interpolating.
\begin{lem}\label{lem:interpolation}[Interpolation property of LPE] Let $f_n$ be an LP($\ell$) estimator with kernel
$K:\matR^d\to \matR_+$ having a singularity at $0$, that is, $\lim_{u \to 0} K(u) = +\infty$, and continuous on $\matR^d \setminus \{0\}$. In particular, there exist $c_0>0$ and $\Delta>0$ such that
\begin{align}\label{k}
    K(u) \geq c_0 \mathbf{1}(\|u\| \leq \Delta), \quad \forall u \in \matR^d.
\end{align}
Assume that $X_1,\dots, X_n$ are distinct points in $\matR^d$ and there exists a constant $\lambda_1>0$ such that  
\begin{align}\label{pos-def}
     \sum_{j=1}^n U\left( \frac{X_j-x}{h}\right) U^{\top }\left( \frac{X_j-x}{h}\right)\mathbf{1}\left(\Big\|\frac{X_j-x}{h}\Big\|\le \Delta\right) \succ \lambda_1 I_{C_{\ell,d}}
\end{align}
for all $x$ in some neighborhood of $X_i$, where $I_{C_{\ell,d}}$ denotes the identity matrix. Then  $f_n(X_i)=Y_i$.
\end{lem}

For $\ell=0$ (corresponding to the Nadaraya-Watson estimator) condition \eqref{pos-def} is trivially satisfied since the expression on the left hand side is a positive scalar for any $x$ in a neighborhood of $X_i$. For general $\ell$, this condition 
is satisfied with high probability if $X_j$'s are distributed with a density bounded away from zero on its support. Indeed, we have the following result. For $\Delta>0$ consider the matrix
\begin{align*}
    \overline B_{nx}: = \frac{1}{n{ h^d}} \sum_{i=1}^n U\left( \frac{X_i-x}{h}\right) U^{\top }\left( \frac{X_i-x}{h}\right)\mathbf{1}\left(\Big\|\frac{X_i-x}{h}\Big\|\le \Delta\right)\in \matR^{C_{\ell,d}\times C_{\ell,d}}.
\end{align*}
\begin{lem}
\label{lem:lem1}
Let $h \le \alpha$, where $\alpha>0$. 
Let Assumption $(A2)$ be satisfied. 
Then, the following holds.

(i) For any $\Delta>0$ there exist constants $\lambda_0(\ell) >0$, $c>0$ independent of $n$ and $x$ and depending only on $\ell, \alpha, \Delta, d, p(\cdot)$ such that
\begin{align*}
    \matP\Big(\inf_{x \in  {\rm Supp}(p)}\lambda_{\min}(\overline B_{nx}) \geq \lambda_0(\ell)\Big) \geq 1 - c(h^{-d^2-d} e^{- nh^d/c} + e^{-n^3h^{2d}/c}),
\end{align*}
where $\lambda_{\min}(\overline B_{nx})$ is the minimal eigenvalue of $\overline B_{nx}$. Moreover, $\lambda_0(\ell)\ge \lambda_0(\ell')$ if $\ell \le \ell'$.

(ii) If $K$ is a kernel 
satisfying \eqref{k}  then there exist constants $\lambda_0'(\ell) >0$, $c'>0$ independent of $n$ and $x$ and depending only on $\ell, \alpha, \Delta, d, p(\cdot)$ such that
\begin{align*}
    \matP\Big(\inf_{x \in  {\rm Supp}(p)}\lambda_{\min}(B_{nx})
    \geq \lambda_0'(\ell)\Big)\geq 1 - c'(h^{-d^2-d} e^{- nh^d/c'} + e^{-n^3h^{2d}/c'}).
\end{align*}
\end{lem}

Note that part (ii) of Lemma \ref{lem:lem1} is an immediate consequence of its part (i) and the fact that $B_{nx} \succ c_0 \overline B_{nx}$ if \eqref{k} holds.
Also, the next corollary follows immediately from Lemmas \ref{lem:interpolation} and \ref{lem:lem1}. 
\begin{cor}\label{cor:interpolation}
Let $f_n$ be an LP($\ell$) with kernel $K:\matR^d\to \matR_+$  having a singularity at $0$, that is, $\lim_{u \to 0} K(u) = +\infty$, and continuous on $\matR^d \setminus \{0\}$. Let $h = \alpha n^{-\frac{1}{2\beta+{d}}}$, where $\alpha,\beta>0$ and let Assumption $(A2)$ be satisfied. 
Then, there exists a constant $c'>0$ such that, with probability at least $1 - c' e^{- A_n/c'}$, where $A_n=n^{\frac{2\beta}{2\beta+d}}$, the LPE $f_n$ is interpolating, that is, $f_n(X_i)=Y_i$ for $i=1,\dots,n$, and $f_n(\cdot)$ is a continuous function on ${\rm Supp}(p)$. Furthermore, the LP($0$) estimator is interpolating with probability 1.
\end{cor}

Note that the kernels  $K(u)=\|u\|^{-a}\mathds{1}(\|u\|\leq 1)$ with $a \in (0, {d}/{2})$ considered in \cite{belkin2018does-bis,belkin2018does} are not continuous on $\matR^d \setminus \{0\}$ and thus do not satisfy the conditions of Lemma \ref{lem:interpolation} and Corollary \ref{cor:interpolation}. On the other hand, these conditions are met for the kernels
$K(u)=\|u\|^{-a}\cos^2(\pi \|u\|/2 )\mathds{1}(\|u\|\leq 1)$ or $K(u)=\|u\|^{-a}(1-\|u\|)_+$ with $a>0$.

\section{Minimax optimal interpolating estimator}

 In this section, we show that for any $\beta>0$, one can construct an interpolating local polynomial estimator reaching the minimax rate $n^{-\frac{2\beta}{2\beta+d}}$ on the Hölder class $\Sigma(\beta,L)$. 

In what follows, we assume that we know a constant $L_0$ such that $|f(x)|\le L_0$ for all $x\in \text{Supp}(p)$. We denote the class of all such functions $f$ by $\mathcal{F}_0$. This assumption is not crucial and can be avoided at the expense of slightly more involved dependence of the result on the noise distribution (see Remark 1 below).  

Let ${f_n}$ be an LP($\ell$) estimator of order $\ell = \lfloor \beta \rfloor$.  Set $\mu := L_0 \lor \max_{1\le i\le n} |Y_i|$  and consider the truncated estimator
\begin{equation}\label{def:truncated_estim}
    \bar f_n(x) = \big[{f_n}(x)\big]_{-\mu}^{\mu},
\end{equation}
where for all $y\in \matR$ and $a \leq b$ the truncation of $y$ between $a$ and $b$ is defined as $[y]_a^b := (y \lor a) \land b$.

\begin{thm}
\label{thm:thm1}
 Let Assumptions (A1)  and (A2) be satisfied. 
Let $f \in \Sigma(\beta,L)$ for $\beta>0, L>0$, and
$|f(x)|\le L_0$ for all $x\in {\rm Supp}(p)$ and a constant $L_0>0$. Consider the estimator $\bar f_n$ defined in \eqref{def:truncated_estim}, where ${f_n}$ is the LP($\ell$) estimator with $\ell = \lfloor \beta \rfloor$,  $h = \alpha n^{-\frac{1}{2\beta+{d}}}$, for some $\alpha>0$, and kernel $K$. 

(i) If $K$ is a compactly supported kernel satisfying \eqref{k}  and  
$\int K^2(u) du <  \infty$ then
\begin{align}
     \matE \left([\bar f_n(x)-f(x)]^2 \right) &\leq C n^{-\frac{2\beta}{2\beta +{ d}}},\quad \forall x \in {\rm Supp}(p), \label{eq:eq10}\\
    \matE\left(\|\bar f_n-f\|_{L_2}^2 \right) &\leq C n^{-\frac{2\beta}{2\beta +{ d}}},\label{eq:eq11}
\end{align}
where $C>0$ is a constant 
depending only on $\beta,L,L_0,d,C_{\xi}, K,p_{\max}, p_{\min}$ and $\alpha$.

(ii)
If, in addition, $\lim_{u \to 0} K(u) = +\infty$ and $K$ is continuous on $\matR^d \setminus \{0\}$,  
then there exists a constant $c'>0$ such that, with probability at least $1 - c' e^{- A_n/c'}$, where $A_n=n^{\frac{2\beta}{2\beta+d}}$, the estimator $\bar f_n$ is interpolating, that is, $\bar f_n(X_i)=Y_i$ for $i=1,\dots,n$, and $\bar f_n(\cdot)$ is a continuous function on ${\rm Supp}(p)$.
\end{thm}

Note that, for the examples of singular kernels given at the end of the previous section, we need $a \in (0, {d}/{2})$ to grant the condition $\int K^2(u) du <  \infty$ required in Theorem \ref{thm:thm1}. Moreover, Shepard kernel $K(u)=\|u\|^{-d}$ does not satisfy the assumptions of Theorem \ref{thm:thm1}.

{\bf Remark 1.}
The value $\max_{1\le i\le n}|Y_i|$ is introduced in the threshold $\mu$ only with the aim to preserve the interpolation property. Inspection of the proof shows that Theorem \ref{thm:thm1}(i) remains valid when $\max_{1\le i\le n}|Y_i|$ is dropped from the definition of $\mu$, so that $\mu=L_0$, but in this case data interpolation is not granted. 
On the other hand, by setting $\mu=2\max_{1\le i\le n}|Y_i|$ it is possible to obtain both items (i) and (ii) of Theorem \ref{thm:thm1} for an estimator that does not require the knowledge of $L_0$. We do not state this result here since we are able to prove it with the constant $C$ in \eqref{eq:eq10} - \eqref{eq:eq11} depending not only on $C_{\xi}$ but also on a tail property of the distribution of $\xi(X)$ given $X$. 

{\bf Remark 2.} Theorem \ref{thm:thm1}(i) completes the existing literature on LPE in the classical setting when the kernel is non-singular. To the best of our knowledge, non-asymptotic bounds on the mean squared error of LPE were not obtained. The previous work was mainly focused on asymptotic properties such as convergence in probability or pointwise asymptotic normality, cf. \cite{stone1980,stone1982optimal,tsybakov1986,fan-gijbels}. For binary $Y\in\{0,1\}$ specific to classification setting, non-asymptotic deviation bounds for LPE were obtained in \cite{audibert2007fast}. However, the techniques of \cite{audibert2007fast} cannot be extended beyond the case of bounded $Y$. 


{\bf Remark 3.} Inspection of the proof shows that Theorem \ref{thm:thm1} extends to kernels  
$K$ that are not necessarily compactly supported. It suffices to assume that the integrals $\int(1+\|u\|^{\beta})K(u)du$ and $\int(1+\|u\|^{2\beta})K^2(u)du$ are finite.

\section{Adaptive interpolating estimator}

In this section, we will use the following assumption on the noise $\xi(X)$.

{\bf Assumption $\mathbf{(A3)}$.} {\it Conditionally on $X=x$, the random variable $\xi(X)$ is a zero-mean $\sigma_\xi$-subgaussian random variable for all $x\in \text{Supp}(p)$.}

We propose an adaptive estimator that does not need the knowledge of $\beta,L,C_{\xi}$, achieves the minimax $L_2$-rate of convergence on classes $\Sigma(\beta,L)$ for all $L>0$ and $\beta \in (0,\betamax]$, where $\betamax>0$ is an arbitrary given value,
and is interpolating with high probability. 
Our adaptive estimator is based on least squares aggregation. We refer to \cite{wegkamp2003model} for the study of such aggregation procedures.

Assume without loss of generality that $n$ is even. We split the sample $\mathcal{D}=\{(X_1,Y_1),\dots,(X_n,Y_n)\}$ into two independent subsamples $\mathcal{D}_1 =\left\{ (X_1,Y_1),\dots,(X_{\frac{n}{2}},Y_{\frac{n}{2}})\right\}$ and $\mathcal{D}_2 = \left\{(X_{\frac{n}{2}+1},Y_{\frac{n}{2}+1}),\dots, (X_n,Y_n)\right\}$, and we proceed in two steps.
\begin{enumerate}
    \item 
    Choose a finite grid $(\beta_j)_{j \in J}$ on the values of $\beta$. Let $f_{n,j}$ denote a LP($\ell_j$) estimator (with $\ell_j=\lfloor \beta_j\rfloor$) based on the subsample $\mathcal{D}_1$ with bandwidth $h=\alpha n^{-\frac{1}{2\beta_j+d}}$, $\alpha>0$, and kernel $K$ satisfying the assumptions of Theorem \ref{thm:thm1}. Set $\mu := L_0 \lor \max_{1\le i\le n/2} |Y_i|$ and construct $|J|$ truncated local polynomial estimators:
\begin{equation}\label{def:estim}
    \bar f_{n,j}(x) = \big[f_{n,j}(x)\big]_{-\mu}^{\mu}, ~~~~ j \in J.
\end{equation}
By Theorem \ref{thm:thm1}, each estimator $\bar f_{n,j}$ is interpolating over $\mathcal{D}_1$ with high probability, and satisfies
    \begin{align}
    \label{eq:eq6}
        \sup_{f \in \Sigma(\beta_j,L)\cap \mathcal{F}_0}\matE_1
        \left[\|\bar f_{n,j}-f\|_{L_2}^2\right] \leq Cn^{-\frac{2\beta_j}{2\beta_j+d}},
    \end{align}
    where $\matE_1$ denotes the expectation with respect to the distribution of $\mathcal{D}_1$.
    \item 
    From the collection $(f_{n,j})_{j\in J}$, we select an estimator $\widetilde{f}_n$ that minimizes the sum of squares over the second subsample $\mathcal{D}_2$, that is, we set $\widetilde f_n = \bar f_{n,\widetilde j}$ with
    \begin{align*}
        \widetilde j \in \underset{j \in J}{\rm argmin} \sum_{k=\frac{n}{2}+1}^n \left( Y_k - \bar f_{n, j}(X_k) \right)^2. 
    \end{align*}
    
    \end{enumerate}
    
    As each of the estimators among $(\bar f_{n,j})_{j\in J}$ is interpolating over $\mathcal{D}_1$, the estimator $\widetilde f_n$  is also interpolating over $\mathcal{D}_1$, but not over $\mathcal{D}_2$. We therefore introduce the estimator {$\widetilde g
    _n$} obtained in the same way as {$\widetilde f_n$} by interchanging $\mathcal{D}_1$ and $\mathcal{D}_2$. Thus, {$\widetilde g_n$} is interpolating over $\mathcal{D}_2$. Next, we define an estimator  interpolating over $\mathcal{D}_1 \cup \mathcal{D}_2$ by combining {$\widetilde f_n$} and {$\widetilde g_n$} as follows. \newpar

For any $x \in \matR^d$ and any set $A\subseteq \matR^d$, denote by $d(x,A) = \inf_{y\in A}\|x-y\|$ the distance between $x$ and $A$. Let $\lambda : \matR^d \rightarrow [0,1]$ be any continuous function such that $\lambda(x) \to 0$ as $d(x,\mathcal{D}_2) \to 0$ and  $\lambda(x) \to 1$ as $d(x,\mathcal{D}_1) \to 0$. For example, take $\lambda(x) = \frac{2}{\pi} \arctan\left(\frac{d(x,\mathcal{D}_2)}{d(x,\mathcal{D}_1)}\right)$ with $\frac{1}{0} =\infty$ and $\arctan(\infty)=1$ by convention. We define our final estimator as
\begin{equation}\label{def:final_estimator}
    \hat {\sf f}_n(x) = \lambda(x) {\widetilde f_n(x)} + (1-\lambda(x)) {\widetilde g_n(x)}.
\end{equation}
\begin{thm}\label{thm:thm3}
Let $n\geq 3$, $\betamax>1$. Consider the grid points $\beta_j$ defined as follows: 
\begin{align*}
    \beta_j = \left(1+ \frac{1}{\log n} \right)^j, ~~ j=-M,\dots,M_{\max},
\end{align*}
where $M=2\left\lfloor \log (n) \log \log (n) \right\rfloor$ and $M_{\max} = M \land \lfloor \log(n) \log (\betamax)\rfloor$. Let Assumptions (A1) and (A3) be satisfied. If kernel $K$ satisfies the assumptions of Theorem \ref{thm:thm1}(i),
then for any $\beta\in (0,\betamax]$ and $L>0$ for the estimator $\hat {\sf f}_n$ defined by \eqref{def:final_estimator} we have 
\begin{align}\label{eq:thm2}
    \sup_{f \in \Sigma(\beta,L)\cap \mathcal{F}_0} \matE\left[\|\hat {\sf f}_n-f\|_{L_2}^2 \right] \leq C n^{-\frac{2\beta}{2\beta+d}},
\end{align}
where $C>0$ is a positive constant depending only on $\beta,L,L_0,d,\betamax,\sigma_\xi, K,$  $p_{\max}, p_{\min}$ and $\alpha$.

If, in addition, kernel $K$ satisfies the assumptions of Theorem \ref{thm:thm1}(ii), then the estimator $\hat {\sf f}_n$ is an interpolating continuous function with probability at least $1 - c''\exp(-n^{\frac{2}{2+d}}/c'')$, where $c''$ is a positive constant depending only on $L,L_0,d,\betamax, K,$ $p_{\max}, p_{\min}$ and $\alpha$. 
\end{thm}


\section{Numerical experiment}

In this section, we report some results of our numerical experiment with singular kernel local polynomial estimators. We ran simulations with various kernels and various regression functions in dimension $d=1$. We present below 
some examples of obtained results for two regression functions:
\begin{align*}
    f(x) = x^3 - x \quad \text{and} \quad
    g(x) = x + \cos(3x).
\end{align*} 
We generated $X_1, \dots, X_n$ according to a uniform law on $\left[-2,2\right]$ with $n=80$. We set, for all $i \in \left[n\right]$, $Y_i = f(X_i) + \varepsilon_i$ or $Y_i = g(X_i) + \varepsilon_i$, where $\varepsilon_i$'s are independent normal random variables with mean 0 and variance $0.5$.
We considered three singular kernels and the rectangular kernel:
\begin{align*}
    K_1(u) &= \rvert u \rvert^{-a} \mathbf{1}(\rvert u \rvert \leq 1), \\
    K_2(u) &= \rvert u \rvert^{-a} \left(1-\rvert u \rvert \right)_{+}^2, \\
    K_3(u) &= \rvert u \rvert^{-a} \cos^2(\pi \rvert u \rvert /2) \mathbf{1}(\rvert u \rvert \leq 1),\\
    K_{\text{rect}}(u) &= \mathbf{1}(\rvert u \rvert \leq 1)
\end{align*}
for various choices of $a\in (0,1/2)$. Below we only present the results for $a=0.2$. 

Both $f$ and $g$ belong to Hölder classes with any smoothness $\beta$. We take $\beta=8$ and we compute LP($\ell$) estimators with $\ell =7$ and with bandwidth $h$ chosen, for each kernel, to minimize the mean squared error  (MSE) over a dense enough grid. For each singular kernel estimator, we also compute its smoothed version (named Smooth LPE), which is a result of applying the running median with a short window to the initial LPE.

The results are presented below.  For comparison, we reproduce in each figure the LPE with rectangular kernel on the right hand graph.
Note that $K_1$ is not continuous on $\matR^d \setminus \{0\}$ and therefore does not satisfy the assumptions of Lemma \ref{lem:interpolation} ensuring the interpolation property. Nevertheless, our simulations show that the corresponding LPE does interpolate the data. 

The tables present the MSE values. We note that they are bigger for  singular kernel estimators than for rectangular kernel ones but not excessively big. It supports the fact that singular kernel LPE achieves the minimax optimal rate, with probably worse constant factor than for its non-singular kernel counterparts.
Reasonable MSE values for singular kernel LPE's are obtained in spite of the fact that visually they are very spiky.
The best results are observed for smoothed singular kernel method that cleans out the small spikes. Finally, note that the MSE values are better for function $f$, which itself is a polynomial,
than for function $g$.

\begin{figure}[H]
    \centering
    \includegraphics[scale=0.45]{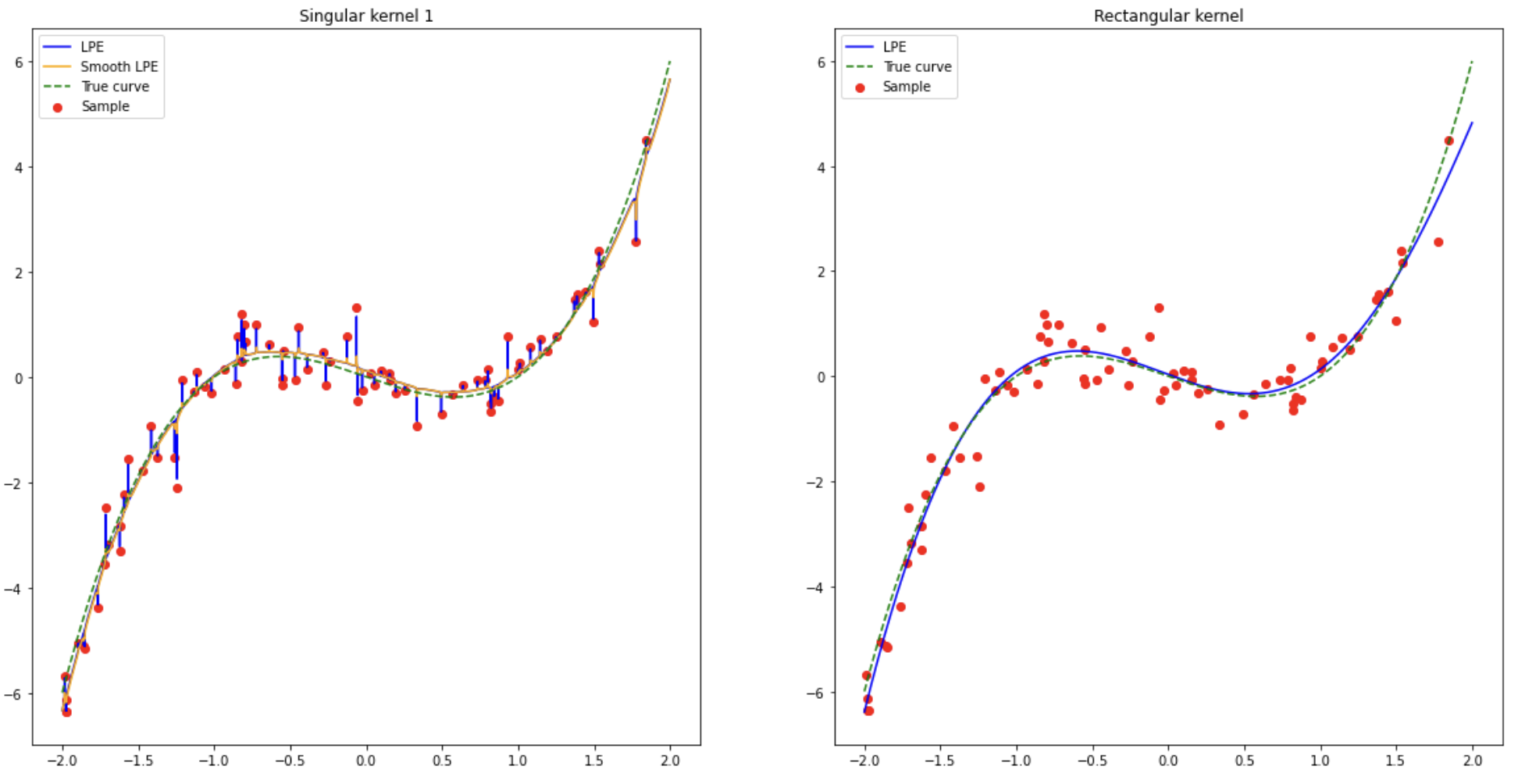}
    \caption{Local polynomial estimator of  regression function $f$ with singular kernel $K_1$ and rectangular kernel.}
  \label{fig:fig_1_1}
  \end{figure}
  
  \begin{figure}[H]
    \centering
    \includegraphics[scale=0.45]{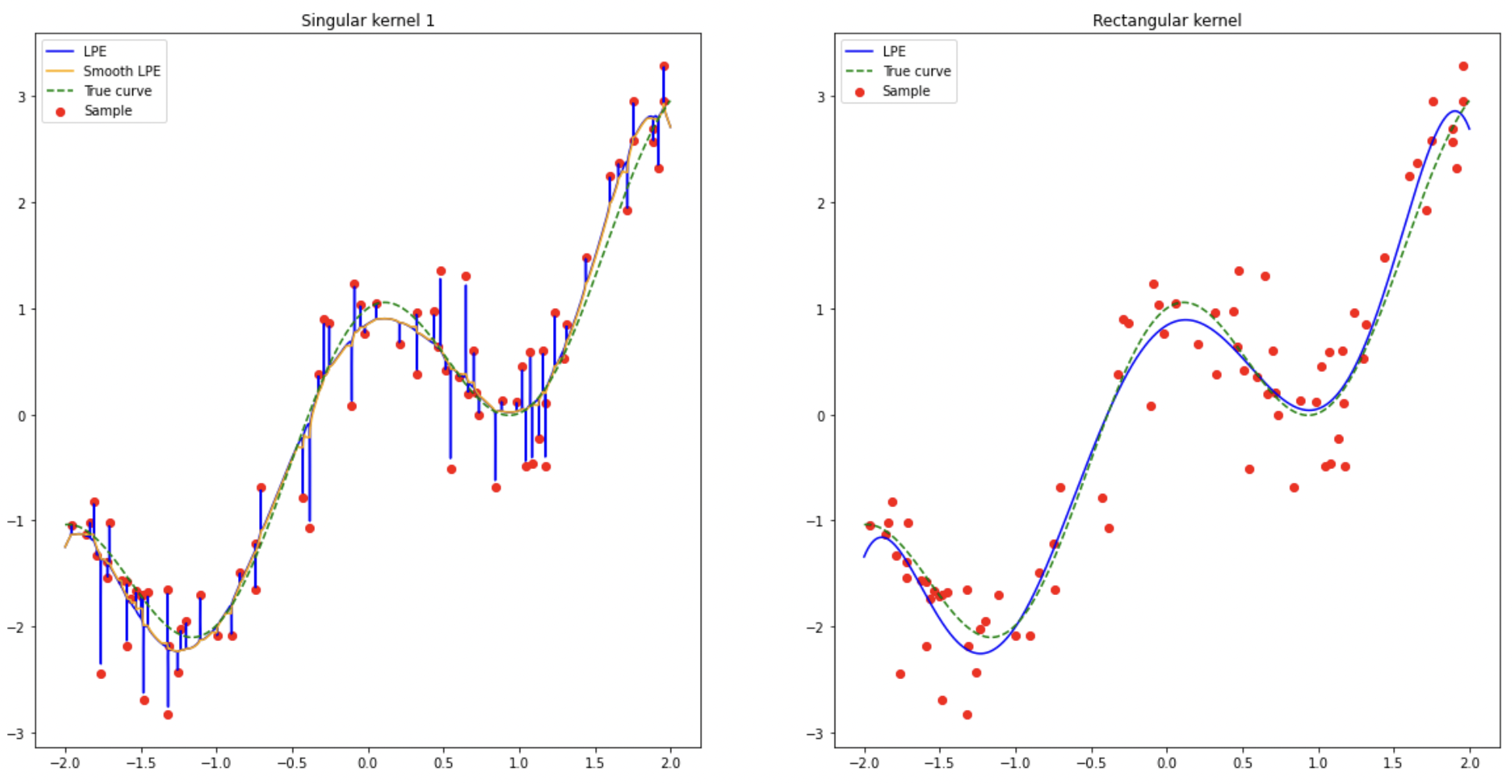}
    \caption{Local polynomial estimator of  regression function $g$ with singular kernel $K_1$ and rectangular kernel.}
  \label{fig:fig_1_2}
  \end{figure}

 \begin{center}
 \begin{tabular}{|l|c|c|c|}
\hline
& Singular kernel $K_1$ & Singular Kernel $K_1$ + Smooth & Rectangular kernel $K_{\text{rect}}$\\
\hline
Function $f$ & 0.0373 & 0.0129 & 0.0129
\\
\hline
Function $g$ & 0.0424 & 0.0144 & 0.0154
\\
\hline
\end{tabular}
\end{center}

\begin{figure}[H]
    \centering
    \includegraphics[scale=0.45]{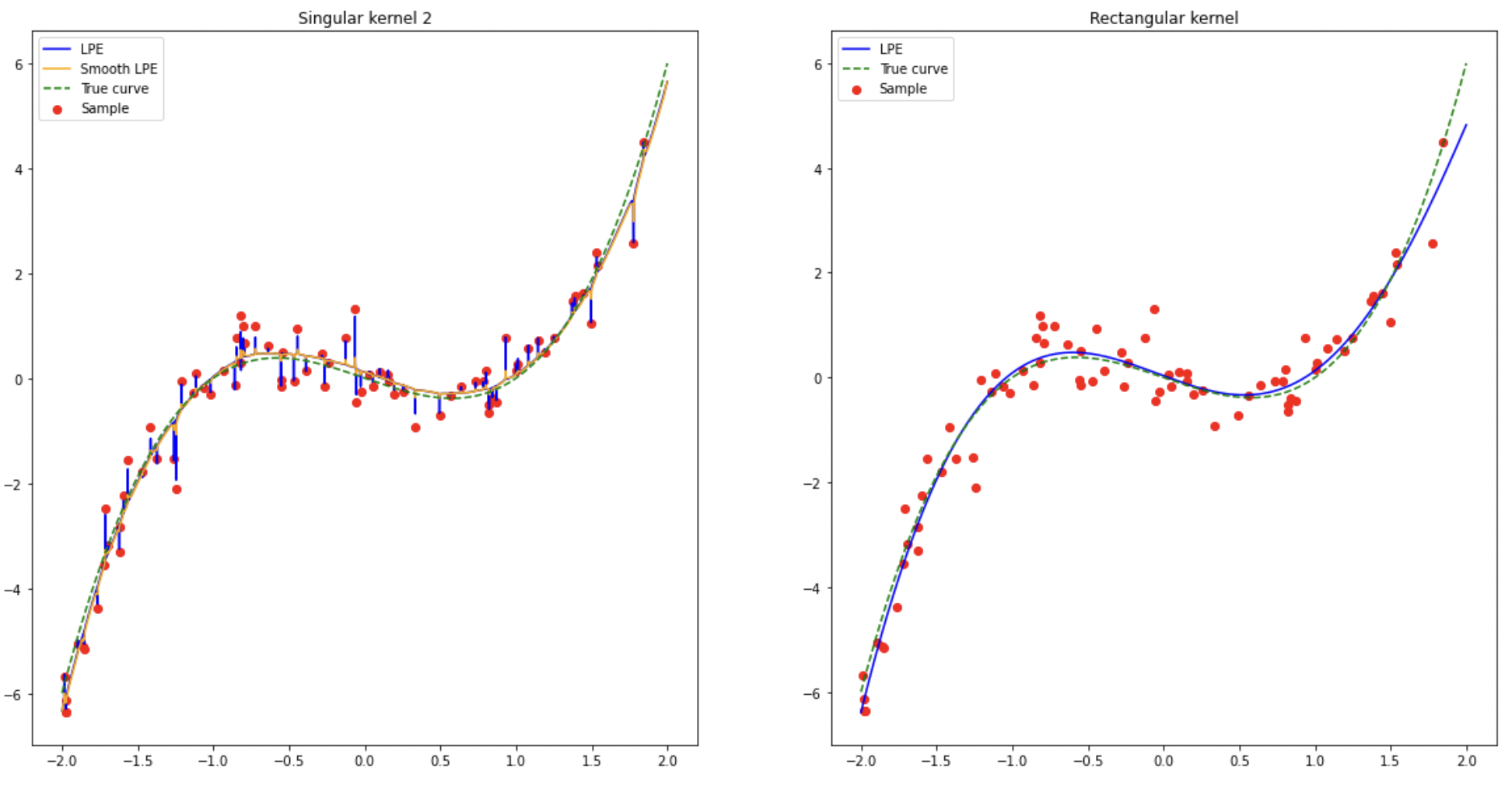}
    \caption{Local polynomial estimator of  regression function $f$ with singular kernel $K_2$ and rectangular kernel.}
  \label{fig:fig_2_1}
  \end{figure}
  
  \begin{figure}[H]
    \centering
    \includegraphics[scale=0.45]{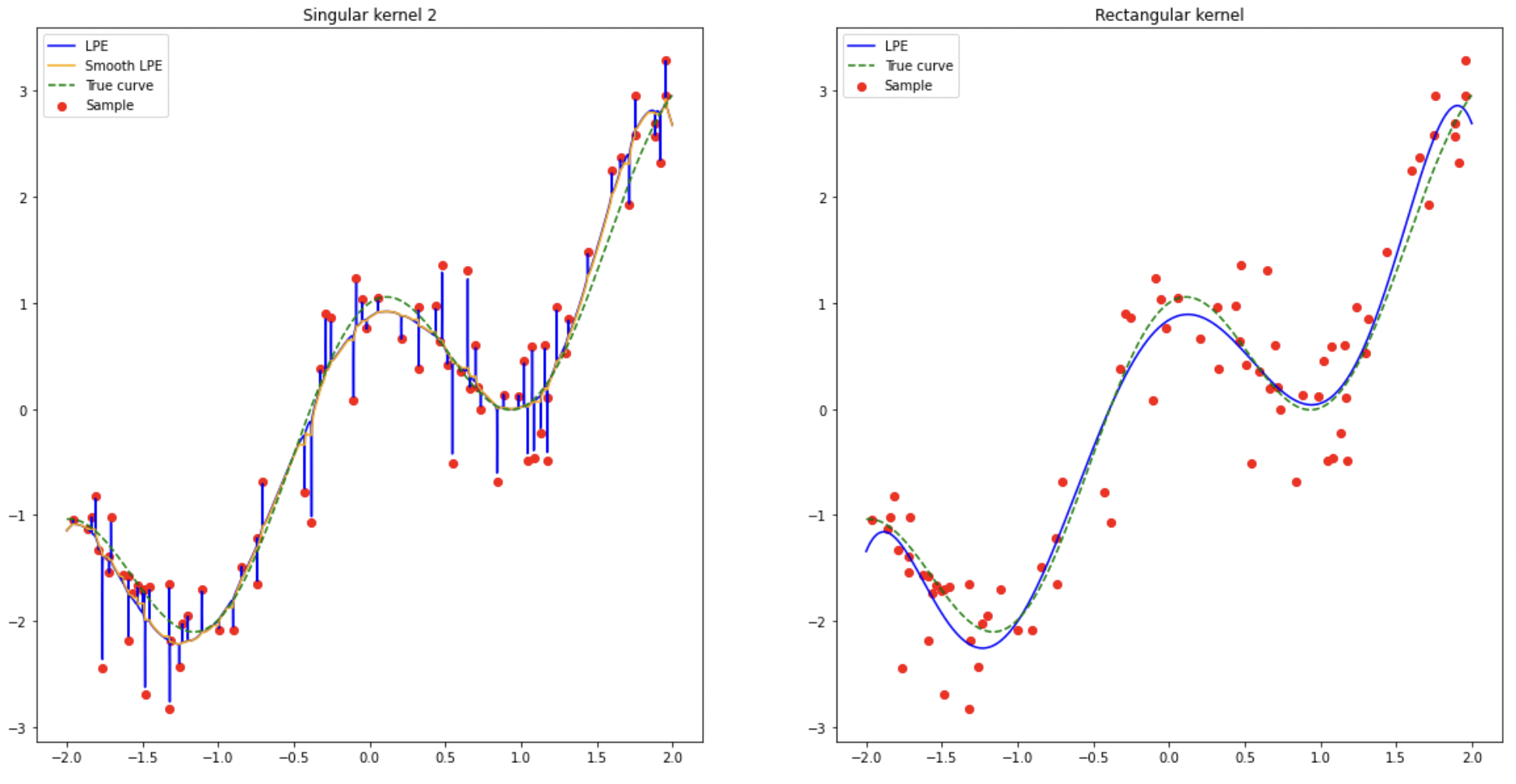}
    \caption{Local polynomial estimator of  regression function $g$ with singular kernel $K_2$ and rectangular kernel.}
  \label{fig:fig_2_2}
  \end{figure}

 \begin{center}
 \begin{tabular}{|l|c|c|c|}
\hline
& Singular kernel $K_2$ & Singular kernel $K_2$ + Smooth & Rectangular kernel $K_{\text{rect}}$\\
\hline
Function $f$ & 0.0383 & 0.0130 & 0.0129
\\
\hline
Function $g$ & 0.0433 & 0.0144 & 0.0154
\\
\hline
\end{tabular}
\end{center}

\begin{figure}[H]
    \centering
    \includegraphics[scale=0.45]{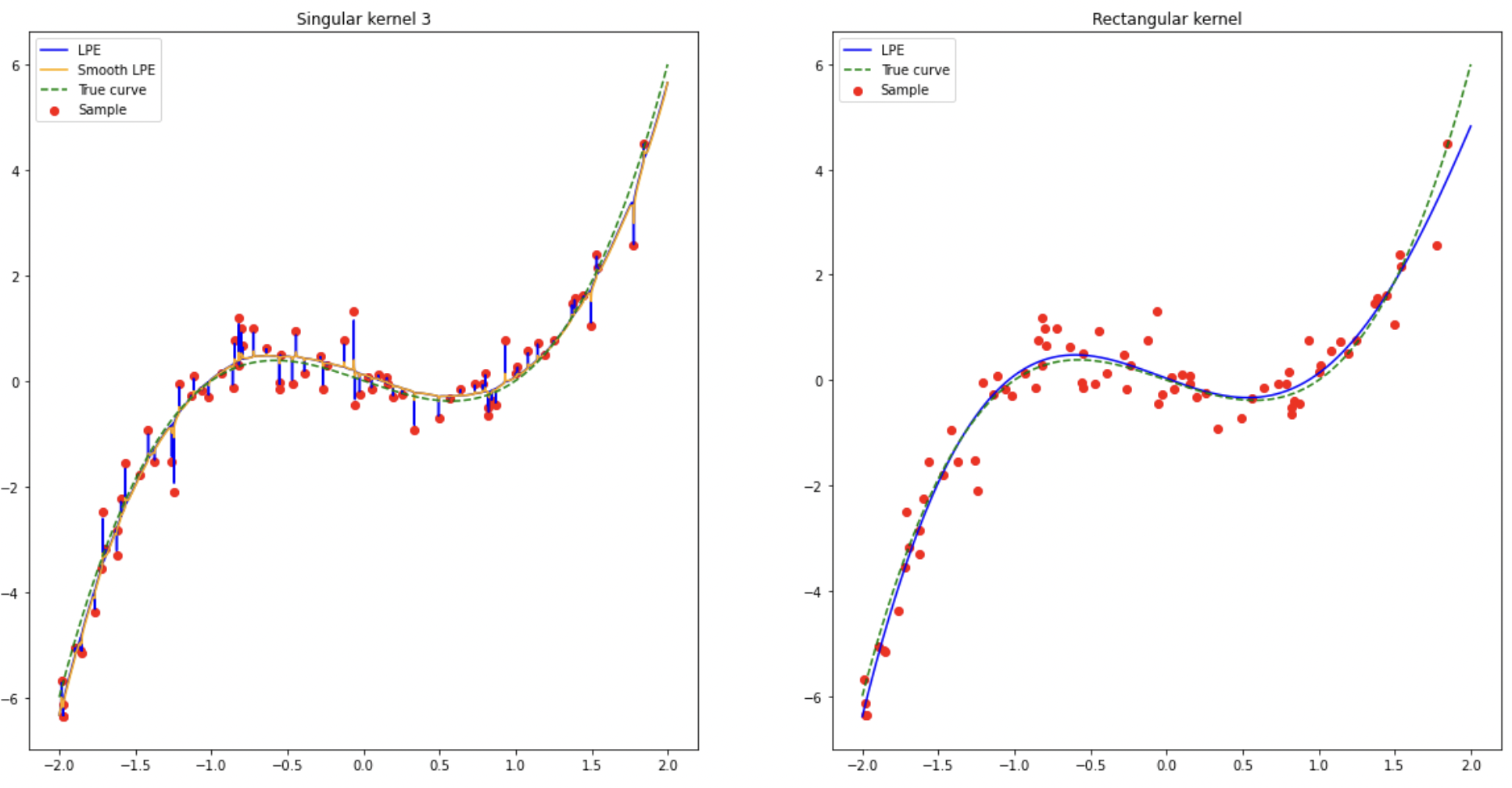}
    \caption{Local polynomial estimator of  regression function $f$ with singular kernel $K_3$ and rectangular kernel.}
  \label{fig:fig_3_1}
  \end{figure}
  
  \begin{figure}[H]
    \centering
    \includegraphics[scale=0.45]{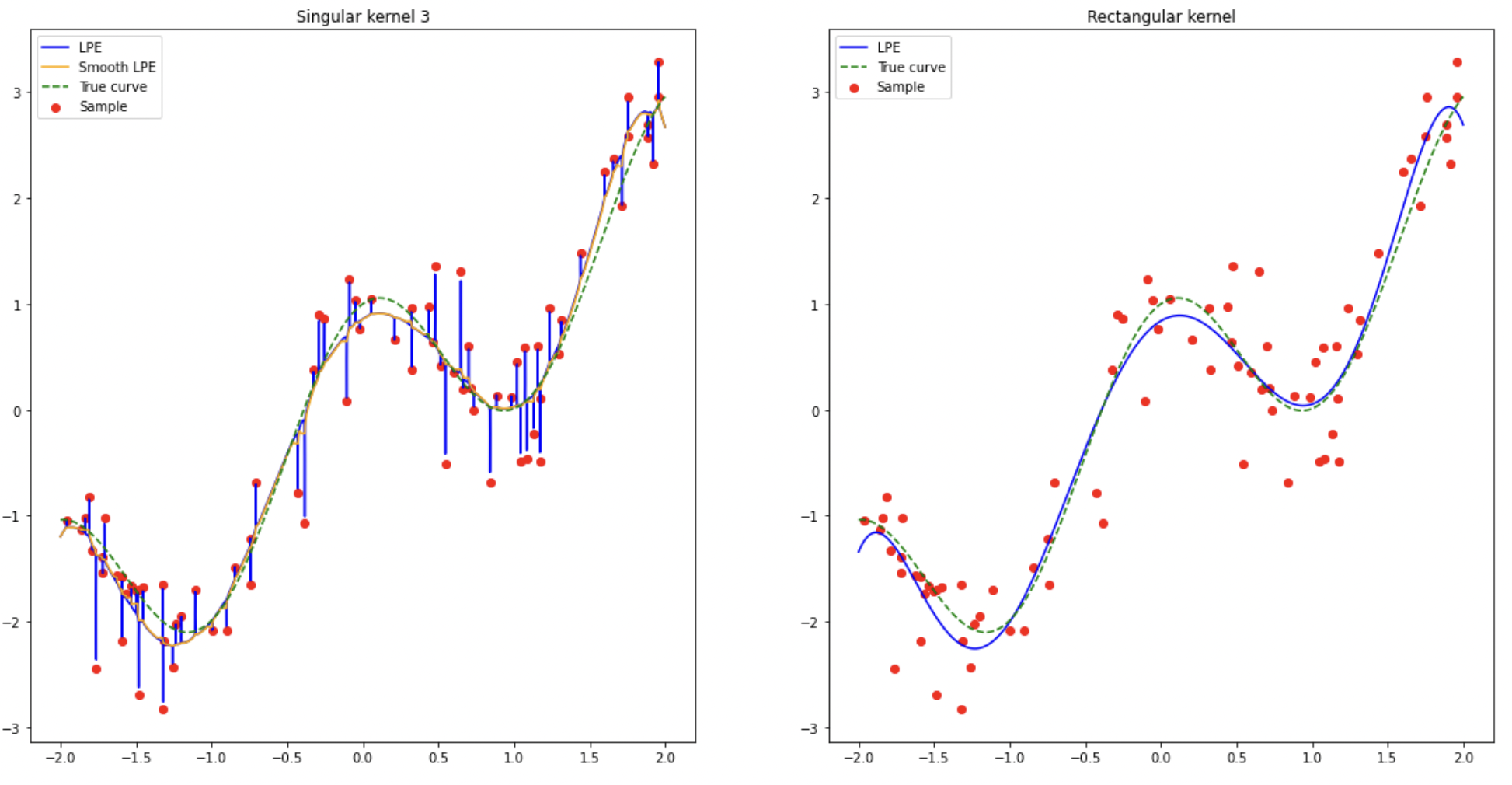}
    \caption{Local polynomial estimator of  regression function $g$ with singular kernel $K_3$ and rectangular kernel.}
  \label{fig:fig_3_2}
  \end{figure}

 \begin{center}
 \begin{tabular}{|l|c|c|c|}
\hline
& Singular kernel $K_3$ & Singular kernel $K_3$ + Smooth & Rectangular kernel $K_{\text{rect}}$\\
\hline
Function $f$ & 0.0373 & 0.0129 & 0.0129
\\
\hline
Function $g$ & 0.0426 & 0.0146 & 0.0154
\\
\hline
\end{tabular}
\end{center}

 \section{Conclusion}

We have shown that local polynomial estimators with singular kernels can achieve minimax optimal rates of convergence (with respect to the mean squared risk) while perfectly interpolating the data, and moreover, can do it adaptively to the the smoothness of the regression function. This seemingly surprising conclusion is indeed not surprising at all because the mean squared risk is used as a criterion. Indeed, by adding "by hand" extremely small spikes to an accurate enough regression estimator we can always get a function interpolating the data and having a reasonably good mean squared risk. Of course, such a construction is very artificial. It makes no sense in practice and it is problematic to achieve adaptation in this way. The miracle of singular kernel LPE is to provide such an effect automatically, including adaptation, as we have outlined above. 
The resulting interpolating estimators have quite a reasonable behavior in terms of mean squared criterion but not in terms of visual criteria. Note that the interpolating procedures developed in different contexts in the recent literature, in particular, in deep learning are analyzed only in terms of mean squared error and expectedly share the same drawback. The difference from our setting is that, in those models, the resulting estimators are not easy to visualize, so that this sort of "spiky" behavior is not made explicit. 

\textbf{Acknowledgments}: This work was supported by the grant of French National Research Agency (ANR) "Investissements d'Avenir" LabEx Ecodec/ANR-11-LABX-0047.

  
  
  

\appendix
\section{Appendix}

\begin{proof}[Proof of Lemma \ref{lem:Holder_nested}]
The result is straightforward if there exists an integer $\ell\ge 0$ such that $\ell<\beta'\le\beta \le \ell+1$. Indeed, for any integer $\ell\ge 0$,
\begin{equation}\label{eq:betabeta1}
  \ell<\beta'\le\beta \le \ell+1 \quad \Longrightarrow \quad \Sigma(\beta,L) \subseteq \Sigma(\beta', L).  
\end{equation}
Thus, it remains to consider the case $\ell<\beta' \le \ell+1<\beta$ for an integer $\ell$. Handling this case will be based on the following embedding: 
\begin{equation}\label{eq:betabeta}
  \Sigma(\beta,L) \subseteq \Sigma(\ell', 2L), \quad \forall \ell'\in \matN \text{ such that } \ell'<\beta.  
\end{equation}
We now prove \eqref{eq:betabeta}. Indeed, let $f \in \Sigma(\beta,L)$ and let $\ell'$ be an integer less than $\beta$. Then, in particular,  $ \max\limits_{0\le s \leq \ell'}\sup\limits_{x\in \Bd} \norm{f^{(s)}(x)} \leq L$. Consider $x,y \in \Bd$ and $h = y-x$. Denote by $h_i$ the $i$th component of $h$ and by $e_i$ the $i$th canonical basis vector in $\matR^d$. Set $k=\ell'-1$. Then for any multi-indices $m_1,\dots,m_{k}\in \matN^d$ we have 
\begin{align*}
D^{m_1+\cdots+m_{k}}f(y)-D^{m_1+\cdots+m_{k}}f(x)&=
\int_0^1 \left\langle\nabla D^{m_1+\cdots+m_{k}}f(x+th), h \right \rangle dt
\\
&=
\int_0^1 \sum_{i=1}^d  D^{m_1+\cdots+m_k+e_i}f(x+th) h_i  dt
\\
&=
 \sum_{i=1}^d \int_0^1 D^{m_1+\cdots+m_{k}+e_i}f(x+th) dt \, h^{e_i}.
\end{align*}
Writing for brevity 
$G_{m_1,\dots,m_k,e_i}(x,h)= \int_0^1 D^{m_1+\cdots+m_{k}+e_i}f(x+th) dt$ we obtain
\begin{align*}
  &  \norm{f^{(k)}(y) - f^{(k)}(x)} 
    =
    \sup_{\substack{\|u_j\|\leq 1,\\j \in [k]}} \Big| \sum_{|m_j|=1, \forall j\in[k]}  \sum_{i=1}^d G_{m_1,\dots,m_k,e_i}(x,h) \, h^{e_i} u_1^{m_1} \dots u_k^{m_k} \Big|
    \\
  & \qquad =
    \|h\| \sup_{\substack{\|u_j\|\leq 1,\\j \in [k]}} \Big| \sum_{|m_j|=1, \forall j\in[k]}  \sum_{i=1}^d G_{m_1,\dots,m_k,e_i}(x,h) \, \left(\frac{h}{\|h\|}\right)^{e_i} u_1^{m_1} \dots u_k^{m_k} \Big|
    \\
   & \qquad \le
    \|h\| \sup_{\substack{\|u_j\|\leq 1,\\j \in [k+1]}} \Big| \sum_{|m_j|=1, \forall j\in[k+1]}   \int_0^1 D^{m_1+\cdots+m_{k+1}}f(x+th) dt \,  u_1^{m_1} \dots u_{k+1}^{m_{k+1}} \Big|
    \\
   & \qquad \leq 
     \|h\| \int_0^1 \sup_{\substack{\|u_j\|\leq 1,\\j \in [k+1]}} \big|f^{(k+1)}(x+th)[u_1,\dots, u_{k+1}]\big| dt
    \\
   &  \qquad\leq 
     \|h\| \sup_{z\in \Bd} \|f^{(k+1)}(z)\|_*
     \le L \|x-y\|,
\end{align*}
which, together with bound $ \max\limits_{0\le s \leq \ell'-1}\sup\limits_{x\in \Bd} \norm{f^{(s)}(x)} \leq L$ implies that $f\in \Sigma(\ell',2L)$. Thus, we have proved \eqref{eq:betabeta}. 

It follows from \eqref{eq:betabeta} that if $\ell<\beta' \le \ell+1<\beta$ for an integer $\ell$
then $\Sigma(\beta,L)\subseteq \Sigma(\ell+1,2L)$, while  taking $\beta=\ell+1$ in \eqref{eq:betabeta1} implies that $\Sigma(\ell+1,2L)\subseteq \Sigma(\beta',2L)$. This proves the lemma when $\ell<\beta' \le \ell+1<\beta$ for an integer $\ell$.
\end{proof}

\begin{proof}[Proof of Lemma \ref{lem:polynomial}]
The result is clear for $\beta \leq 1$. Assume that $\beta>1$ and fix some $x,y \in \Bd$. By Taylor expansion, there exists $c \in (0,1)$ such that
\begin{align*}
    f(x) = \sum_{0\leq |k| \leq \ell-1} \frac{1}{k!} D^k\! f(y)(x\!-\!y)^k + \sum_{ |k| = \ell} \frac{1}{k!} D^k\! f(y\!+\!c(x\!-\!y))(x\!-\!y)^k,
\end{align*}
and
\begin{align*}
    \left|f(x) \hspace{-1mm} - \hspace{-2mm} \sum_{|k| \leq \ell} \frac{1}{k!} D^k\! f(y)(x\!-\!y)^k\right| 
    & = 
    \left| \sum_{|k| = \ell} \frac{1}{k!} \big[D^k\! f\big(y\!+\!c(x\!-\!y)\big)-D^k\! f(y)\big](x\!-\!y)^k \right|.
\end{align*}
By a standard combinatorial argument, it is not hard to check that, for any $h, z\in \matR^d$,
\begin{align*}
f^{(k)}(z)[h]^k:
    &= 
    \sum_{|m_1| = \ldots = |m_{\ell}| = 1} D^{m_1 + \ldots + m_{\ell}}f(z) h^{m_1 + \ldots + m_{\ell}}
    =
    \sum_{|k| = \ell} \frac{\ell !}{k!}D^{k}f(z) h^{k}\enspace.
\end{align*}
It follows that
\begin{align}\label{eq:lem2}
&\left| \sum_{|k| = \ell} \frac{1}{k!} \big[D^k\! f\big(y\!+\!c(x\!-\!y)\big)-D^k\! f(y)\big](x\!-\!y)^k \right|
\\ 
  & \qquad =  \frac{1}{\ell!}
  \left| f^{(\ell)}\big(y\!+\!c(x\!-\!y)\big)[x\!-\!y]^{\ell} - f^{(\ell)}(y) [x\!-\!y]^{\ell} \right|
  \nonumber
  \\ 
  & \qquad \le  \frac{1}{\ell!}
  \left\| f^{(\ell)}\big(y\!+\!c(x\!-\!y)\big) - f^{(\ell)}(y)  \right\|_* \|x\!-\!y\|^\ell
  \nonumber
    \\
    & \qquad \leq \frac{L}{\ell!} \|x\!-\!y\|^\ell \|c(x\!-\!y)\|^{\beta-l} \leq \frac{L}{\ell!} \|x\!-\!y\|^\beta.
    \nonumber
\end{align}
\end{proof}

\begin{proof}[Proof of Lemma \ref{lem:interpolation}]
In this proof, we fix $i \in \left[ n\right]$, and
our aim is to prove that $\lim_{x\to X_i}f_n(x)=Y_i$. Let $\mathcal{V}$ be the neighborhood of $X_i$ where \eqref{pos-def} holds. Since $X_1,\dots,X_n$ are distinct, we assume w.l.o.g. that $\mathcal{V}$ does not contain $(X_j)_{j\ne i}$. Due to conditions \eqref{k} and \eqref{pos-def}, we have that $B_{nx}\succ 0$ for all $x$ in $\mathcal{V}_{-}:=\mathcal{V} \setminus \{X_i\}$. Thus, for all $x\in \mathcal{V}_{-}$ the vector $\hat{\theta}_n(x)$ is the unique solution of \eqref{eq:eq1}, and $f_n(x)$ is given by \eqref{eq:eq4a}:
\begin{align*}
    &\hat{\theta}_n(x)= \underset{\theta \in \matR^{C_{\ell,d}}}{\rm argmin} \ \sum_{i=1}^n \left[Y_i - \theta^{\top} U\left(\frac{X_i-x}{h} \right) \right]^2 K \left(\frac{X_i-x}{h} \right), \\
    &f_n(x) = U^\top (0) \hat{\theta}_n(x). 
\end{align*}
Define $g_i(x) = \left( Y_i - \hat{\theta}_n(x)^\top U\left(\frac{X_i-x}{h}\right) \right)^2$.  
First, we prove by contradiction that $\lim \limits_{x \rightarrow X_i} g_i(x) = 0$ for any $i \in \left[n\right]$. 
Indeed, suppose that $\lim_{x \rightarrow X_i} g_i(x) \neq 0$. Then, there is a sequence $(x_k)_k$ in $\matR^d$ converging to $X_i$ as $k\to \infty$ such that $\lim_{k \to \infty} g_i(x_k) = + \infty$ or $\lim_{k \to \infty} g_i(x_k)= {\rm const} >0$.
In both cases, 
\begin{align}\label{lim-gk}
\lim_{k \to \infty} \sum_{j=1}^n g_j(x_k) K\left( \frac{X_j-x_k}{h} \right) = +\infty
\end{align}
since the kernel $K$ has a singularity at 0. 
On the other hand, the definition of $\hat{\theta}_n(x_k)$ implies that, for any $k$ and any $\theta_* \in \matR^{C_{\ell,d}}$,
$$
\sum_{j=1}^n g_j(x_k) K\left( \frac{X_j-x_k}{h} \right) \le 
\sum_{j=1}^n \left(Y_j - \theta_*^\top U\left( \frac{X_j-x_k}{h}\right) \right)^2 K\left( \frac{X_j-x_k}{h}\right).
$$
In particular, for $\theta_*^\top = \left( Y_i \quad 0 \dots 0 \right)$ we have
\begin{align*}
    \sum_{j=1}^n \left(Y_j - \theta_*^\top U\left( \frac{X_j-x_k}{h}\right) \right)^2 K\left( \frac{X_j-x_k}{h}\right) =& \sum_{j=1}^n (Y_j - Y_i)^2 K\left( \frac{X_j-x_k}{h}\right) \\
    =& \sum_{j \neq i} (Y_j - Y_i)^2 K\left( \frac{X_j-x_k}{h}\right) \\
    \underset{k \rightarrow + \infty}{\rightarrow} &\sum_{j \neq i} (Y_j - Y_i)^2 K\left( \frac{X_j-X_i}{h}\right) < +\infty,
\end{align*}
which is in contradiction with \eqref{lim-gk}. 
Therefore, for any $i \in \left[n\right]$ we have $\lim \limits_{x \rightarrow X_i }g_i(x) = 0$. 

A similar argument yields that $\lim \sup \limits_{x \rightarrow X_i }g_j(x)
< +\infty$ for any $j\ne i$. Indeed, if for some $j\ne i$ this relation does not hold then there is a sequence $(x_k)_k$ in $\matR^d$ converging to $X_i$ as $k\to \infty$ such that $\lim_{k \to \infty} g_j(x_k) = + \infty$. It implies \eqref{lim-gk}, which is not possible as shown above.

Next, we prove that $\|\hat{\theta}_n(x)\|$ is bounded for all $x$ in a neighborhood of $X_i$.
Since $\lim \limits_{x \rightarrow X_i} g_i(x) = 0$, and for any $j\ne i$ we have $\lim \sup\limits_{x \rightarrow X_i} g_j(x) < + \infty$ the values $g_j(x)$  are bounded for all $j \in \left[ n\right]$ and all $x$ in a neighborhood of $X_i$. We will further denote this neighborhood by $\mathcal{V}'$. It follows that $\varphi_j(x) = \hat{\theta}_n(x)^\top U\left( \frac{X_j-x}{h} \right)$, $j=1,\dots,n,$ are bounded for $x\in \mathcal{V}'$ and thus the sum $\sum_{j=1}^n \varphi_j^2(x)$ is bounded as well. On the other hand, by assumption \eqref{pos-def}, for all $x\in \mathcal{V}_{-}$,
\begin{align*}
    \sum_{j=1}^n \varphi_j^2(x) &\ge  \sum_{j=1}^n \hat{\theta}_n(x)^\top U\left( \frac{X_j-x}{h}\right) U^\top\left( \frac{X_j-x}{h}\right) \mathbf{1}\left( \left\| \frac{X_j-x}{h} \right\|\leq \Delta\right)\hat{\theta}_n(x) \\
    &\ge  \lambda_1 \| \hat{\theta}_n(x) \|^2, 
\end{align*}
where $\lambda_1>0$. It follows that $\| \hat{\theta}_n(x) \|$ is bounded for all $x\in \mathcal{V}'\cap \mathcal{V}_{-}$. 

Let $\hat{\theta}_{n,(1)}(x)=f_n(x)$ denote the first component of $\hat{\theta}_{n}(x)$ and $\hat{\theta}_{n,(2)}(x)$ the vector of its remaining $C_{\ell,d}-1$ components, so that $\hat{\theta}_{n}(x)^\top = \left(\hat{\theta}_{n,(1)}(x),  \hat{\theta}_{n,(2)}(x)^\top \right)$. Recall that the first component of $U(u)$ is equal to 1 for all 
$u\in \matR^d$. Denote by $U_{(2)}(u)$ the vector of its remaining $C_{\ell,d}-1$ components, so that $U(u)^\top = \left( 1, U_{(2)}(u)^\top\right)$. 
With this notation, the relation $\lim \limits_{x \rightarrow X_i} g_i(x) = 0$ proved above can be written as:
\begin{align*}
    g_i(x) &= \left( Y_i - \hat{\theta}_{n,(1)}(x) - \hat{\theta}_{n,(2)}(x)^\top U_{(2)}\left(\frac{X_i-x}{h}\right) \right)^2 
    \underset{x \rightarrow X_i}{\rightarrow}  0.
\end{align*}
Since $\| \hat{\theta}_n(x) \|$ is bounded for  $x\in \mathcal{V}'\cap \mathcal{V}_{-}$ we get that $| \hat{\theta}_{n,(1)}(x)|$ and $\| \hat{\theta}_{n,(2)}(x) \|$ are also bounded for  $x\in \mathcal{V}'\cap \mathcal{V}_{-}$. The definition of $U(u)$ implies the convergence $\lim \limits_{x \rightarrow X_i}\| U_{(2)}\left(\frac{X_i-x}{h}\right)  \|=0$. It follows that 
\begin{align*}
    \hat{\theta}_{n,(2)}(x)^\top U_{(2)}\left(\frac{X_i-x}{h}\right) \underset{x \rightarrow X_i}{\rightarrow}  0
\end{align*}
and therefore
\begin{align*}
    \hat{\theta}_{n,(1)}(x) \underset{x \rightarrow X_i}{\rightarrow} Y_i,
\end{align*}
which concludes the proof since $\hat{\theta}_{n,(1)}(x)=f_n(x)$.
\end{proof}

\hfill

\begin{proof}[Proof of Lemma \ref{lem:lem1}]
We prove only part (i) of the lemma since part (ii) is its immediate consequence.
We have
$$ \overline{B}_{nx} = \frac{1}{n{ h^d}}\sum_{i=1}^n U \left(\frac{X_i-x}{h} \right)U^{\top} \left(\frac{X_i-x}{h} \right) \mathbf{1}\left( \frac{\| X_i - x \|}{\Delta} \leq h\right) $$
and, for any $\lambda_0>0$,
\begin{align}\nonumber
    &\matP\left(\inf_{x \in  {\rm Supp}(p)} \hspace{-3mm} \lmin(\overline{B}_{nx}) < \lambda_0\right) = \matP\left(\inf_{x \in  {\rm Supp}(p)} \inf_{\| v \| = 1} v^{\top} \overline{B}_{nx} v < \lambda_0 \right) \\
    & \leq \matP \left(\inf_{x \in  {\rm Supp}(p)} \inf_{\| v \| = 1} \! v^{\top} \overline{B}(x) v - \hspace{-4mm} \sup_{x \in  {\rm Supp}(p)} \hspace{-4mm} \| \overline{B}_{nx} \!-\! \overline{B}(x) \|_{\infty} < \lambda_0  \right)\label{eq:proofL4:1}
\end{align}
where $\overline{B}(x) := \matE(\overline{B}_{nx})$. Set $S(x,h,\Delta) = \left\{u\in \Bd(0,\Delta): x + uh \in {\rm Supp}(p)\right\}$. Then we have
\begin{align*}
    v^{\top} \overline{B}(x) v &= {\frac{1}{h^d}} \int \left[ v^{\top} U \left(\frac{z-x}{h} \right)\right]^2 \mathbf{1} \left( \left\| \frac{z-x}{h}\right\|\leq \Delta\right)p(z) \mathrm{d}z \\
    & \ge p_{\min} v^\top \left[\int_{S(x,h,\Delta)} U(u)U(u)^\top du\right] v
     \\
    & \ge p_{\min} v^\top \left[\int_{S(x,\alpha,\Delta)} U(u)U(u)^\top du\right] v,
\end{align*}
where for the last inequality we used the fact that $ S(x,\alpha,\Delta) \subset  S(x,h,\Delta)$ since $h  \leq \alpha$ and $\text{Supp}(p)$ is a convex set. Notice that $S(x,\alpha,\Delta)$ is
also a convex set and it is not reduced to one point $x$ as $\text{Supp}(p)$ is a convex set with positive Lebesgue measure. Thus, $S(x,\alpha,\Delta)$ is of infinite cardinality for any $x\in \text{Supp}(p)$.

Denote by $S_d(0,1)$ the unit sphere in $\matR^d$
centered at 0. Note that, for fixed $\Delta$ and $\alpha$, the function 
\begin{align*}
\left\{
\begin{array}{ccl}
    \text{Supp}~p \times S_d(0,1) &\longrightarrow & \matR\\
    (x,v) &\mapsto& v^\top \left[\int_{S(x,\alpha,\Delta)} U(u)U(u)^\top du\right] v
\end{array}
\right.
\end{align*}
is continuous and defined on a compact set.  Therefore, it attains its minimum at some $(x_0,v_0)$, where $x_0\in \text{Supp}(p)$ and  $\|v_0\|=1$. We argue now that the value of this minimum is positive. Indeed, it is clearly non-negative, and if it were $0$ we would have:
\begin{align}\label{eq:zero}
    0=v_0^\top U(u) = \sum_{|k|\leq \ell} v_0(k) \frac{u^k}{k!}, \quad \forall u \in S(x_0,\alpha,\Delta).
\end{align}
As observed above, $S(x_0,\alpha,\Delta)$ is a set of infinite cardinality. On the other hand, the expression in \eqref{eq:zero} is a polynomial in $u$, so that for $v_0 \ne 0$ it can vanish only in a finite number of points. Thus, \eqref{eq:zero} is impossible. It follows that 
$$
\lambda_1(\ell): =\min_{v\in S_d(0,1),x \in \text{Supp}(p)}v^\top \left[\int_{S(x,\alpha,\Delta)} U(u)U(u)^\top du\right] v >0 .
$$
Next, note that the vector $U(u) = U_{\ell}(u)$ depends on $\ell$, and that for $\ell \leq \ell'$ and any fixed $x$, the matrix $\int_{S(x,\alpha,\Delta)} U_\ell(u)U_\ell(u)^\top du$ is an extraction of the matrix $\int_{S(x,\alpha,\Delta)} U_{\ell'}(u)U_{\ell'}(u)^\top du$.
Hence, the smallest eigenvalue of the former matrix is necessarily not less than that of the latter. Thus, $\lambda_1(\ell)\ge \lambda_1(\ell')$ for $\ell \leq \ell'$.

Setting $\lambda_0=\lambda_0(\ell): =p_{\min} \lambda_1(\ell)/2$ and using \eqref{eq:proofL4:1} we find:
\begin{align}\label{eq:proofL4:2}
   \matP \left(\inf_{x \in \text{Supp}(p)} \lmin(\overline B_{nx}) < \lambda_0\right) \leq \matP\left( \sup_{x \in  \text{Supp}(p)}\| \overline{B}_{nx} - \overline{B}(x)\|_{\infty} >  \lambda_0\right).
\end{align}
It remains now to bound the probability on the right hand side of \eqref{eq:proofL4:2}.

By Assumption (A2), the convex compact set $\text{Supp}(p)$ is included in  $\Bd=\mathcal{B}_d(0,1)$. For $\varepsilon > 0$, let  $\{x_1,\dots,x_N\} \subset \mathcal{B}_d^N$ be the minimal $\varepsilon$-net on $\Bd$ in the Euclidean metric. Then we have:  
\begin{align*}
    \sup_{x \in \text{Supp}(p)}\|\overline{B}(x) - \overline{B}_{nx}\|_{\infty} &\leq \sup_{x \in \Bd} \min_{1\leq k \leq N} \|\overline{B}(x) - \overline{B}(x_k)\|_\infty \\
    & + \max_{1\leq k \leq N} \|\overline{B}(x_k) - \overline{B}_{nx_k}\|_{\infty}  + \sup_{\substack{x,x' \in \Bd,\\ \| x - x'\| \leq \varepsilon}} \|\overline{B}_{nx} - \overline{B}_{nx'}\|_{\infty}.
\end{align*}
Thus,
\begin{align}
\label{eq:eq3}
&\matP\left(\sup_{x\in \text{Supp}(p)}\| \overline{B}_{nx} - \overline{B}(x)\|_{\infty} > \lambda_0 \right) \le P_1+P_2 +P_3, \ \ \text{where} 
\\
& P_1=\matP\left(\sup_{x \in \Bd}\min_{1\leq k \leq N} \|\overline{B}(x) - \overline{B}(x_k)\|_\infty > \frac{\lambda_0}{3}\right), \nonumber \\
& P_2=
\matP\left(\max_{1\leq k \leq N} \|\overline{B}(x_k) - \overline{B}_{nx_k}\|_{\infty} > \frac{\lambda_0}{3}\right), \nonumber \\ 
& P_3= \matP\Bigg(\sup_{\substack{x,x' \in \Bd:\\ \| x - x'\| \leq \varepsilon}} \|\overline{B}_{nx} - \overline{B}_{nx'}\|_{\infty}>\frac{\lambda_0}{3}\Bigg). \nonumber
\end{align} 
In the rest of the proof, we control the terms $P_1,P_2,P_3$.

{\it Control of $P_2$.} Since all norms in the space of $C_{\ell,d} \times C_{\ell,d}$ matrices  are equivalent there exists a constant $c_1 >0$ depending only on $\ell, d$ such that, for all $k\in \{1,\dots,N\}$,
\begin{align*}
    \| \overline{B}(x_{k})-\overline{B}_{nx_{k}} \|_{\infty} \leq c_1 \max_{1\leq i,j \leq C_{\ell,d}} \rvert b_{nx_{k}}(i,j) - b_{x_{k}}(i,j) \rvert
\end{align*}
where $b_{nx_{k}}(i,j)$ and $b_{x_{k}}(i,j)$ are the elements of $\overline{B}_{nx_{k}}$ and $\overline{B}(x_{k})$, respectively. 
Then, for any $k\in\{1,\dots,N\}$,
\begin{align*}
    \matP\left(\|  \overline{B}(x_{k})-\overline{B}_{nx_k} \|_{\infty} > \frac{\lambda_0}{3}\right) \leq C_{\ell,d}^2 \max_{1\leq i,j\leq C_{\ell,d}} \matP\left(\rvert b_{nx_k}(i,j)-b_{x_{k}}(i,j)\rvert > \frac{\lambda_0}{3c_1} \right). 
\end{align*}
We recall that $b_{x_{k}}(i,j) = \matE \left[b_{nx_k}(i,j)\right]$. Setting $s=s^{(i)}$ and $r=s^{(j)}$ we have
\begin{align*}
    b_{nx_k}(i,j) = \frac{1}{n{ h^d}} \sum_{m=1}^n \frac{(X_m - x_k)^s}{h^s s!}\frac{(X_m - x_k)^r}{h^r r!} \mathbf{1} \left( \left\| \frac{X_m-x_k}{h} \right\| \leq \Delta \right).
\end{align*}
This is a sum of $n$ i.i.d. random variables, each of which is bounded in absolute value by ${\frac{C}{n h^d}}$ and has variance not exceeding ${\frac{C}{n^2 h^d}}$, where $C>0$ is a constant depending only on $\ell,d,\Delta$. By Bernstein's inequality, 
\begin{align*}
    \matP\left(\rvert b_{nx_k}(i,j)-b_{x_{k}}(i,j)\rvert > \frac{\lambda_0}{3c_1} \right) \leq 
    2  \exp(-c_2n h^d),
\end{align*}
where $c_2 >0$ only depends on $\ell,d,\Delta$ and not on $n, k, i,j$. It follows from the above inequalities and the union bound that
\begin{align}
\label{eq:eq12a}
    P_2 \leq 
    2 N C_{\ell,d}^2 \exp(-c_2n h^d).
\end{align}

{\it Control of $P_3$.} For any $x,x'\in \Bd$,
\begin{align*}
    \overline{B}_{nx} - \overline{B}_{nx'} = \frac{1}{nh^d}\sum_{i=1}^n&\Big[
    U\left(\frac{X_i-x}{h}\right)U^\top \left(\frac{X_i-x}{h}\right)\mathbf{1}\left(
    \left\|\frac{X_i-x}{h}\right\|\leq \Delta\right) 
    - \\
    &U\left(\frac{X_i-x'}{h}\right)U^\top \left(\frac{X_i-x'}{h}\right)\mathbf{1}\left(
    \left\|\frac{X_i-x'}{h}\right\|\leq \Delta\right)\Big].
\end{align*}
For any $u \in \matR^d$ consider the matrix
\begin{equation}
    V(u) = U(u) U^\top(u) \mathbf{1}\{\|u\|\leq \Delta\}.
\end{equation}
Notice that $U(u)\in\matR^{C_{\ell,d}}$ is Lipschitz continuous in $u$ on the ball $\mathcal{B}_d(0,\Delta)$ since the components of vector $U(u)$ are polynomials in $u$.  Thus, there exists a  constant $\widetilde L>0$ depending only on $\ell$ and $d$ such that for any $u,u' \in \matR^d$,  if either $\|u\|\leq \Delta, \|u'\|\leq \Delta$ or $\|u\| > \Delta, \|u'\| > \Delta$, then
\begin{align*}
 \big\| V(u)-V(u') \big\|_{\infty} \leq \widetilde L \|u-u'\| ,
\end{align*}
and if $(u,u')$ belongs to the set $$\widetilde\Delta:=\{(u,u'):\|u\|\leq \Delta, \|u'\|> \Delta\}\cup \{(u,u'):\|u\| > \Delta, \|u'\| \le \Delta\}$$ then
\begin{align*}
    \big\| V(u)-V(u') \big\|_{\infty} \leq \widetilde L,
\end{align*}
taking $\widetilde L \geq \max\limits_{\|u\| \leq\Delta}\left\|U(u) U(u)^\top\right\|_\infty$. It follows that
\begin{align*}
    \big\| V(u)-V(u') \big\|_{\infty} \leq \widetilde L\Big\{\|u-u'\|  + \mathbf{1}((u,u') \in \widetilde\Delta)\Big\},\numberthis \label{eq:V_Lipschitz}
\end{align*}
which implies the bound
\begin{align*}
    \hspace{-3mm} \|\overline{B}_{nx} - \overline{B}_{nx'}\|_\infty \leq \frac{\widetilde L}{h^{d+1}}\|x-x'\| + \frac{\widetilde L}{nh^d}\text{Card}\Big\{i \in [n] : X_i \in \widetilde\Delta(x,x',h\Delta)\Big\},
\end{align*}
where we denote by $\widetilde\Delta(x,x',h\Delta)$ the symmetric difference $\mathcal{B}_d(x,h\Delta) \triangle \mathcal{B}_d(x',h\Delta)$. Thus,
\begin{equation}
\label{eq:eq13a}
    \sup_{\substack{x,x' \in \Bd:\\ \| x - x'\| \leq \varepsilon}}\|\overline{B}_{nx} - \overline{B}_{nx'}\|_\infty \leq \frac{\widetilde L \varepsilon}{h^{d+1}} + \frac{\widetilde L}{nh^d} 
    \sup_{\substack{x,x' \in \Bd:\\ \| x - x'\| \leq \varepsilon}}
    \sum_{i=1}^n \mathbf{1} \big(X_i \in \widetilde\Delta(x,x',h\Delta)\big),
\end{equation}
If $\|x-x'\| \leq \varepsilon$ then
$$\widetilde\Delta(x,x',h\Delta)\subseteq \{z:h\Delta< \|z-x\|\le h\Delta+\varepsilon\} \cup \{z:h\Delta< \|z-x'\|\le h\Delta+\varepsilon\}.$$
Therefore,  for $\|x-x'\| \leq \varepsilon$ we have  $\big|\widetilde \Delta(x,x',h\Delta)\big| \leq C_* h^{d-1} \varepsilon$,  where we denote by $|S|$ the Lebesgue measure of a measurable set $S \subset \matR^d$, and $C_*>0$ is a constant depending only on $\Delta$ and $d$. 
Set $\varepsilon= c_0 h^{d+1}$, where the constant $c_0$ satisfies $0< c_0\le \displaystyle{\frac{\lambda_0}{6 \widetilde L}}$. Then for $\|x-x'\| \leq \varepsilon$ we get $\matP(X_1\in \Delta(x,x',h\Delta)) \le p_{\max} C_* c_0 h^{2d}$. Choose $c_0$ small enough (and depending only on $\ell, d, p_{\min} p_{\max}, \Delta$) to satisfy
$p_{\max} C_* c_0 \alpha^{d}\le \displaystyle{\frac{\lambda_0}{12 \widetilde L} }$. Consider the random event 
\begin{equation*}
    \mathcal{A} = \Big\{ \sup_{\substack{x,x' \in \Bd:\\ \| x - x'\| \leq \varepsilon}}\sum_{i=1}^n \mathbf{1} \big(X_i \in \widetilde\Delta(x,x',h\Delta)\big) \leq A\Big\},
\end{equation*}
where 
$A=\displaystyle{\frac{\lambda_0}{6\widetilde L} nh^d}$. Due to the choice of $c_0$ and the fact that $h\le \alpha$ the bound $\matP(X_1\in \Delta(x,x',h\Delta)) \le A/2$ holds whenever $\|x-x'\| \leq \varepsilon$. Hence, 
\begin{align}\label{eq:pp}
    \matP\big(\overline{\mathcal{A}}\big) &\le
    \matP\Big\{ \sup_{\substack{x,x' \in \Bd:\\ \| x - x'\| \leq \varepsilon}}\Big|\frac1n \sum_{i=1}^n \mathbf{1} \big(X_i \in \widetilde\Delta(x,x',h\Delta)\big)-\matP\big(X_1\in \Delta(x,x',h\Delta)\big)\Big| \geq A/2\Big\}.
\end{align}
The class of all balls in $\matR^d$ has a VC-dimension at most $d+2$, cf. Corollary 13.2 in \cite{devroye-gyorfi-lugosi}. Consequently, the class of all intersections of two balls in $\matR^d$  has a VC-dimension at most $Cd$
where $C>0$ is an absolute constant \cite{vdVaart-Wellner}. This allows us to apply the Vapnik-Chervonenkis inequality to bound the probability in \eqref{eq:pp}. Indeed, we can use the decomposition  
\begin{align}\nonumber
\mathbf{1} \big(X_i \in \widetilde\Delta(x,x',h\Delta)\big)&=\mathbf{1} \big(X_i \in \mathcal{B}_d(x,h\Delta) \big) + \mathbf{1} \big(X_i \in \mathcal{B}_d(x',h\Delta) \big)
\\
& \quad - 2 \cdot \mathbf{1} \big(X_i \in \mathcal{B}_d(x,h\Delta) \cap \mathcal{B}_d(x',h\Delta) \big)\label{eq:3terms}
\end{align}
and bound from above the probability in \eqref{eq:pp} by the three probabilities corresponding to the three terms on the right hand side of \eqref{eq:3terms}.
Applying the Vapnik-Chervonenkis inequality \cite[Theorem 12.5]{devroye-gyorfi-lugosi} to each of these probabilities we get
\begin{align*}
    \matP\big(\overline{\mathcal{A}}\big) &\le c_3
    n^{c_3} \exp(- nA^2/128)\le c_3 n^{c_3} \exp(- c_4 n^3h^{2d}),
\end{align*}
where $c_3>0, c_4>0$ are constants depending only on $d,\ell, p(\cdot),\Delta$. On the other hand, due to \eqref{eq:eq13a} and the definitions of $\varepsilon$  and $A$, on the event $\mathcal{A}$ we have
\begin{align*}
\sup_{\substack{x,x' \in \Bd:\\ \| x - x'\| \leq \varepsilon}}
    \|\overline{B}_{nx} - \overline{B}_{nx'}\|_\infty \leq  \frac{\lambda_{0}}{3}.
\end{align*}
Thus, we have proved that
\begin{equation}\label{eq:P2}
    P_3\le c_3n^{c_3} \exp(- c_4 n^3h^{2d}).
\end{equation}

{\it Control of $P_1$.} 
Fix $x \in \Bd$ and  let $k \in \{1,\dots,N\}$ be such that $\| x - x_k\| \leq \varepsilon$. Using \eqref{eq:V_Lipschitz} we obtain 
\begin{align*}
    \|\overline{B}(x) - \overline{B}(x_k)\|_\infty &\leq \frac{1}{h^d}\int_{\matR^d} \left\|V\left(\frac{z-x}{h}\right)- V\left(\frac{z-x_k}{h}\right)\right\|_\infty p(z)dz 
    \\
    & \leq \frac{\widetilde L}{h^d} \int_{\matR^d}\Big[\frac{\epsilon}{h} + \mathbf{1}(z \in \widetilde\Delta\big(x,x_k,h\Delta)\big)\Big] p(z)dz 
    \\
    & \leq \widetilde L \varepsilon \left(\frac{1}{h^{d+1}} + \frac{C_*p_{\max}}{h}\right) ~~ \text{ (since $\big|\widetilde \Delta(x,x_k,h\Delta)\big| \leq C_* h^{d-1} \varepsilon$)}
    \\
    &= \widetilde L c_0 \left(1 + C_*p_{\max} h^d\right) \le \widetilde L c_0 \left(1 + C_* p_{\max}\alpha^d\right) < \frac{\lambda_0}{3}
\end{align*}
provided that $c_0$ is chosen small enough (depending only on $\ell, d, p(\cdot), \Delta,\alpha$). Thus, $P_1=0$ under this choice of $c_0$. Combining this remark with \eqref{eq:proofL4:2}, \eqref{eq:eq12a} and \eqref{eq:P2} we conclude that
\begin{align*}
\matP \left(\inf_{x \in \text{Supp}(p)} \lmin(\overline B_{nx}) < \lambda_0\right) 
 \leq 
    2 N C_{\ell,d}^2 \exp(-c_2n h^d) + c_3 n^{c_3} \exp(- c_4 n^3h^{2d}).
\end{align*}
Recall that the cardinality $N$ of the minimal $\varepsilon$-net on the ball $\Bd=\Bd (0,1)$ satisfies $N \leq \left(\frac{2}{\varepsilon} + 1\right)^d$. The result of the lemma now follows by observing that under our choice of $\varepsilon$ we have $N \leq  C h^{-d^2-d}$, where the constant $C>0$  depends only on $\ell, d, p(\cdot), \Delta,\alpha$.
\end{proof}
\hfill

In the proof of Theorem \ref{thm:thm1} below, we will use the fact that an LP($\ell$) estimator reproduces the polynomials of degree $\leq \ell$ for all $x \in \matR^d$ such that $B_{nx} \succ0$. We state this property in the next proposition. The proof is omitted. It follows the same lines as the proof of Proposition 1.12 in \cite{tsybakov2008introduction} dealing with the case $d=1$. 
\begin{prop}
\label{prop:prop1}
Let $x \in \matR^d$ such that $B_{nx} \succ 0$ and let $Q$ be a polynomial of degree $\leq \ell$. Then the LP($\ell$) weights $W_{ni}$ are such that
\begin{align*}
    \sum_{i=1}^n Q(X_i) W_{ni}(x) = Q(x).
\end{align*}
In particular,
\begin{align}\label{eq:prop1}
    \sum_{i=1}^n W_{ni}(x) = 1 \text{ and } \sum_{i=1}^n (X_i - x)^k W_{ni}(x) = 0 \text{ for } |k| \leq \ell.
\end{align}
\end{prop}

\bigskip

\begin{proof}[Proof of Theorem \ref{thm:thm1}]
Part (ii) of the theorem follows from Corollary \ref{cor:interpolation}.
Also, note that \eqref{eq:eq11} is an immediate consequence of  \eqref{eq:eq10} and Assumption (A2). 
Therefore, we need only to prove \eqref{eq:eq10}. 

Fix $x\in \text{Supp}(p)$ and 
define the random events $\mathcal{E}_0=\big\{x\not \in  \{X_1,\dots,X_n\}\big\}$, and
\begin{align*}
    &\mathcal{E} = \left\{ \lmin(B_{nx}) \geq \lambda_0' \right\}\cap \mathcal{E}_0.
\end{align*}
where $\lambda_0'=\lambda_0'(\ell)$ is a constant from Lemma \ref{lem:lem1} that does not depend on $n$ and $x$. From Assumption (A2) we get that $\matP(\mathcal{E}_0)=1$. This and Lemma \ref{lem:lem1} with our choice of $h$ yield: 
\begin{equation}\label{eq:probaE}
    \matP(\overline{\mathcal{E}})\le c' e^{- A_n/c'}, 
\end{equation}
where $A_n=n^{\frac{2\beta}{2\beta+d}}$ and $c'>0$ does not depend on $x$ and $n$. 

Since $|\bar f_n(x)|\le \mu = \max_{1\le i\le n}|Y_i| 
\lor L_0$
we obtain
\begin{align*}
    \matE\left(\left[ \bar f_n(x) - f(x)\right]^2 \right) &
    \leq \matE\big(\left[ \bar f_n(x) - f(x)\right]^2 \mathbf{1}({ \mathcal{E}}) \big) + \matE\big([L_0 + \mu]^2  \mathbf{1}(\overline{\mathcal{E}})\big) 
    \\
    &
    \leq \matE\big(\left[ f_n(x) - f(x)\right]^2 \mathbf{1}({\mathcal{E}}) \big) + \matE\big([L_0 + \mu]^{2+\delta}\big)^{\frac{2}{2+\delta}} \matP(\overline{\mathcal{E}})^{\frac{\delta}{2+\delta}}, 
\end{align*}
where we have used Hölder's inequality and the fact that  $|\bar f_n(x)-f(x)| \leq | f_n(x)-f(x)|$ for all $x \in \text{Supp}(p)$.
Next, 
\begin{align*}
    \matE\big([L_0 + \mu]^{2+\delta}\big)\le 
    \matE\big([{2L_0 + \max_{1\le i\le n}|\xi(X_i)|}]^{2+\delta}\big)\le C\big[1+ n \matE\big( |\xi(X_1)|^{2+\delta}\big)\big].
\end{align*}
Using this inequality and Assumption (A1) we get
\begin{align}\label{eq:a1}
    \matE\left(\left[ \bar f_n(x) - f(x)\right]^2 \right) 
    &
    \leq \matE\big(\left[ f_n(x) - f(x)\right]^2 \mathbf{1}({\mathcal{E}}) \big) + C n^{\frac{2}{2+\delta}} \matP(\overline{\mathcal{E}})^{\frac{\delta}{2+\delta}}, 
\end{align}
We now bound the main term $\matE\big(\left[ f_n(x) - f(x)\right]^2 \mathbf{1}(\mathcal{E}) \big)$ on the right hand side of~\eqref{eq:a1}. Writing for brevity $\matE[\cdot|X_1,\dots,X_n]=\widetilde \matE[\cdot]$ we have
\begin{align}\label{eq:a2}\nonumber
    \matE\left(\left[ f_n(x) - f(x)\right]^2 \mathbf{1}(\mathcal{E}) \right) 
    & \leq 2\matE \left( \left(f_n(x) - \widetilde\matE[f_n(x)]\right)^2 \mathbf{1}(\mathcal{E}) \right)
    \\
    & \quad + 2\matE \left( \left(\widetilde\matE[f_n(x)] - f(x)\right)^2 \mathbf{1}(\mathcal{E}) \right).
\end{align}    
We analyze separately the two terms (bias and variance terms) on the right hand side of \eqref{eq:a2}.  

{\it Bound on the variance term.} On the event $\mathcal{E}$ we have
\begin{align*}
    \widetilde \matE [ f_n(x)] = \sum_{i=1}^n f(X_i) W_{ni}(x),
\end{align*}
where
$$
W_{ni}(x) = \frac{1}{n{ h^d}} U^\top(0) B^{-1}_{nx}U\left(\frac{X_i-x}{h}\right)K\left(\frac{X_i-x}{h}\right).
$$
Thus, using Assumption (A1) the variance term can be bounded as follows:
\begin{align*}
    \matE \left( \left( f_n(x) - \widetilde\matE[ f_n(x)]\right)^2 \mathbf{1}(\mathcal{E}) \right) &= \matE \left( \left( \sum_{i=1}^n  \xi(X_i) W_{ni}(x)\right)^2 \mathbf{1}(\mathcal{E}) \right) \\
    & = \matE\left(\sum_{i=1}^n \matE\left[ \xi^2(X_i)\big| X_i\right] W_{ni}^2(x) \mathbf{1}(\mathcal{E})\right) 
    \leq C \sigma^2(x),
\end{align*}
where
\begin{align*}
    \sigma^2(x) = \matE \left( \sum_{i=1}^n W_{ni}^2(x)\mathbf{1}(\mathcal{E})\right).
\end{align*}
In what follows, we assume w.l.o.g. that $\text{Supp}(K)\subseteq \Bd$.
On the event $\mathcal{E}$, we have  $\|B^{-1}_{nx}v\| \leq \|v\|/\lambda_0'$ for any $v \in \matR^{C_{\ell,d}}$. This inequality and the fact that $\|U(0)\|=1$ imply 
\begin{align*}
    \rvert W_{ni}(x)\rvert &\leq \frac{1}{n{ h^d}} \left\|B^{-1}_{nx}U\left(\frac{X_i-x}{h}\right)K\left(\frac{X_i-x}{h}\right) \right\| 
    \\
    &\leq \frac{1}{n{ h^d}\lambda_0'}\left\| U\left(\frac{X_i-x}{h}\right) \right\| K\left(\frac{X_i-x}{h} \right) 
    \\
    &\leq \frac{1}{n{ h^d}\lambda_0'}K\left(\frac{X_i-x}{h}\right)\sqrt{ \sum_{0\le |s|\leq \ell} \frac{1}{(s!)^2}} 
    \qquad \text{(since $\text{Supp}(K)\subseteq \Bd$)}
    \\
    &{\leq} \frac{c_5}{n{ h^d}}K\left(\frac{X_i-x}{h}\right) = :\zeta_i,
\end{align*}
where $c_5>0$ is a constant that does not depend on $n$ and $x$. Using Assumption (A2) and the compactness of the support of $K$ we get 
\begin{align}\label{eq:a4}
    &\matE(\zeta_1^2) \leq \frac{c_5^2 p_{\max}}{n^2h^d}\int K^2(u)\mathrm{d}u \leq \frac{C}{n^2h^d}, 
    \\
    &\matE(\zeta_1) \leq \frac{c_5 p_{\max}}{n}\int K(u)\mathrm{d}u \leq \frac{C}{n} \left(\int K^2(u)\mathrm{d}u \right)^{1/2} \leq \frac{C}{n}.\label{eq:a5}
\end{align}
It follows that
\begin{align*}
    \sigma^2(x) \leq \matE\left(\sum_{i=1}^n \zeta_i^2\right) \leq \frac{C}{nh^d}
\end{align*}
and
\begin{align}\label{eq:a3}
\matE \left( \left( f_n(x) - \widetilde\matE[ f_n(x)]\right)^2 \mathbf{1}(\mathcal{E}) \right) & \leq \frac{C}{nh^d}.
\end{align}

{\it Bound on the bias term.} On the event $\mathcal{E}$ we have 
\begin{align*}
    \widetilde \matE[ f_n(x)] - f(x) &= \sum_{i=1}^n f(X_i)W_{ni}(x) - f(x) \\
    &= \sum_{i=1}^n [f(X_i)-f(x)]W_{ni}(x),
\end{align*}
so that the bias term in \eqref{eq:a2} can be written as
\begin{align*}
    \matE \left( \left( \widetilde \matE[ f_n(x)] - f(x)\right)^2 \mathbf{1}(\mathcal{E}) \right) = \matE \left(\left[\sum_{i=1}^n [f(X_i) - f(x)]W_{ni}(x) \right]^2 \mathbf{1}(\mathcal{E}) \right) = : b^2(x). 
\end{align*}
Using \eqref{eq:prop1} and the Taylor expansion of $f$ we get that for some $\tau_i \in [0,1]$,
\begin{align*}
    \sum_{i=1}^n [f(X_i)-f(x)]W_{ni}(x) = \sum_{i=1}^n \sum_{\rvert k \rvert = \ell} \frac{D^kf(x+\tau_i(X_i-x))}{k!}(X_i-x)^k W_{ni}(x) 
    \\
    = \sum_{i=1}^n \sum_{\rvert k \rvert = \ell} \frac{(D^kf(x+\tau_i(X_i-x))-D^kf(x)}{k!}(X_i-x)^k W_{ni}(x).
\end{align*}
Since $f$ belongs to $\Sigma(\beta,L)$ we can apply \eqref{eq:lem2}, which yields
\begin{align*}
    b^2(x)  &\leq \matE\left[\left(\sum_{i=1}^n{\frac{L}{\ell!} \|X_i\!-\!x \|^{\beta}}\rvert W_{ni}(x)\rvert\right)^2 \mathbf{1}(\mathcal{E}) \right]  \\
    &= \matE\left[\left(\sum_{i=1}^n{\frac{L}{\ell!} \|X_i\!-\!x \|^{\beta}}\,\rvert W_{ni}(x)\rvert\, \mathbf{1}(\|X_i-x\| \!\leq\! h)\!\right)^2 \hspace{-2mm} \mathbf{1}(\mathcal{E}) \right] \text{ (as supp}(K) \subset \Bd)\\
    &\leq \matE \left[\left(\sum_{i=1}^n {\frac{L}{\ell!}  h^{\beta}}\rvert W_{ni}(x) \rvert \right)^2 \mathbf{1}(\mathcal{E}) \right].
\end{align*}
As $| W_{ni}(x)|\le \zeta_i$ we further get 
\begin{align*}
    b^2(x) &\leq Ch^{2\beta}\matE\left[\left(\sum_{i=1}^n \zeta_i \right)^2\right]
    =
    C h^{2\beta} \left[ \sum_{i=1}^n \matE(\zeta_i^2)+ \sum_{i\ne j}^n \matE(\zeta_i)\matE(\zeta_j)\right]
     \\
    & =  C h^{2\beta} \left[ n \matE(\zeta_1^2)+n(n-1)\matE(\zeta_1)^2\right]
    \le C h^{2\beta},
\end{align*}
where the last inequality follows from \eqref{eq:a4}, \eqref{eq:a5} and the fact that $h=\alpha n^{-\frac{1}{2\beta+d}}$. Combining this bound on $b^2(x)$ with \eqref{eq:probaE}, \eqref{eq:a1}, \eqref{eq:a2} and \eqref{eq:a3} we finally obtain 
\begin{align*}    
    \matE\left(\left[\bar f_n(x) - f(x)\right]^2 \right)&\leq 
    C\left(\frac{1}{nh^d}+h^{2\beta}+ n^{\frac{2}{2+\delta}} e^{- n^a/C}
    \right),
\end{align*}
where $a=\frac{2\beta}{2\beta+d}$.
Since $h=\alpha n^{-\frac{1}{2\beta+d}}$ the desired bound \eqref{eq:eq10} follows. 
\end{proof}

\hfill

\begin{proof}[Proof of Theorem \ref{thm:thm3}]
If $K$ satisfies the assumptions of Theorem \ref{thm:thm1}(ii) then each estimator $\bar f_{n,j}$ is interpolating on $\mathcal{D}_1$ with probability at least $$1 - C\exp(-n^{-{2\beta_j}/(2\beta_j+d)}/C)\ge 1 - C\exp(-n^{-\frac{2}{2+d}}/C)$$ if $\beta_j>1$, and with probability 1 if $0<\beta_j\le 1$. Hence all of them are simultaneously interpolating with probability at least $$1 - C M_{\max}\exp(-n^{-\frac{2}{2+d}}/C)\geq 1 - C'\exp(-n^{-\frac{2}{2+d}}/C'),$$
and the same holds true for the estimator $\widetilde f_n$. Analogously, the estimator $\widetilde g_n$ is  interpolating on $\mathcal{D}_2$ with the same probability. These remarks and the definition of $\hat {\sf f}_n$ in \eqref{def:final_estimator} ensure that $\hat {\sf f}_n$ is interpolating on the whole sample $\mathcal{D}$ with probability at least $1 - 2 C'\exp(-n^{-\frac{2}{2+d}}/C')$.

We now prove the bound \eqref{eq:thm2}. First, we show that such a bound holds for the estimator $\widetilde f_n$.  Set $B \!=\! L_0 + \mu$. Then $\|\bar{f}_{n,j} - f \|_\infty \leq B$ for all $j\!=\!-M, \dots, M_{\max}$,
where $\|\cdot \|_\infty$ denotes the $L_\infty$-norm on ${\rm Supp}(p)$. Fix the subsample $\mathcal{D}_1$. Then $\bar{f}_{n,j}$'s become fixed functions, and applying Theorem 2.1 in \cite{wegkamp2003model} with $a\!=\!1$, $\lambda_j \!=\! 0, ~ \forall j\!=\!-M, \dots, M_{\max},$ and $K \!=\! M+M_{\max} + 1 \leq C \log^2(n)$, we get 
\begin{align}\label{eq:wegkamp}
    \matE_2\left[\|\widetilde f_n - f\|_{L_2}^2  \right] \leq 2 \min_{-M \leq j \leq M_{\max}}  \|\bar f_{n,j} - f\|_{L_2}^2  + \frac{C(B^2 \log \log n + \log^2 (n))}{n},
\end{align}
where we denote by $\matE_2$  the expectation over the distribution of the sample $\mathcal{D}_2$, and we have used the fact that $M_{\max}\le M$.
Note that under Assumption (A3) we have $\matE_1 (B^2) \leq C \log n$ (see, e.g., Lemma 1.6 in \cite{tsybakov2008introduction}). Therefore, taking the expectations over  $\mathcal{D}_1$ on both sides of \eqref{eq:wegkamp} we get 
\begin{align}
    \matE_1\matE_2\left[\|\widetilde f_n - f\|_{L_2}^2\right] \leq 2 \min_{-M \leq j \leq M_{\max}} \matE_1 \left[\|\bar f_{n,j} - f\|_{L_2}^2\right] +  C \frac{\log^2 (n) \log \log n}{n}.
\label{eq:eq8}
\end{align}
Assume now that $\beta \in [\beta_j, \beta_{j+1}]$ for some $j \in \{-M, \dots, M_{\max}-1\}$. Lemma \ref{lem:Holder_nested} implies that $\Sigma(\beta,L) \subseteq \Sigma(\beta_j, 2L)$. Hence, using
\eqref{eq:eq6}, we obtain:
\begin{align*}
    \sup_{f \in \Sigma(\beta,L)\cap \mathcal{F}_0} \matE_1\left[ \|\bar f_{n,j}-f\|_{L_2}^2\right] &\leq \sup_{f \in \Sigma(\beta_j,2L)\cap \mathcal{F}_0} \matE_1\left[ \|\bar f_{n,j}-f\|_{L_2}^2\right] \leq Cn^{-\frac{2\beta_j}{2\beta_j+d}}. \numberthis \label{eq:eq9}
\end{align*}
Combining \eqref{eq:eq8} and \eqref{eq:eq9} we get that, for $\beta \in [\beta_j, \beta_{j+1}]$,
\begin{align}
    \sup_{f \in \Sigma(\beta,L)\cap \mathcal{F}_0}\matE_1\matE_2\left[\|\widetilde{f}_n-f\|_{L_2}^2\right] &\leq  C n^{-\frac{2\beta_j}{2\beta_j+d}}.  \label{eq:eq10x}
\end{align}
 Notice that if $\beta \in [\beta_j, \beta_{j+1}]$ for some $j \in \{-M, \dots, M_{\max}-1\}$ then 
$$ n^{-\frac{2\beta_j}{2\beta_j+d}}\le e
n^{-\frac{2\beta}{2\beta+d}}. $$ 
Indeed,
\begin{align*}
    \frac{\beta}{2\beta+d} - \frac{\beta_j}{2\beta_j+d} &\leq \frac{\beta_{j+1} - \beta_j}{(2\beta+d)(2\beta_j + d)} 
     = \frac{\beta_j}{(2\beta_j+d)(2\beta + d) \log n}\\
     &\leq \frac{\beta}{(2\beta_j+d)(2\beta + d) \log n} \leq \frac{1}{2 \log n}.
\end{align*}
The case $\beta \in [\beta_{M_{\max}},\betamax]$ is treated analogously. These remarks and \eqref{eq:eq10x} imply thta for each $\beta \in [\beta_{-M},\betamax]$ there exists a constant $C>0$ such that 
\begin{align}\label{eq:a6}
    \sup_{f \in \Sigma(\beta,L)\cap \mathcal{F}_0}\matE \left[\|\widetilde{f}_n-f\|_{L_2}^2\right] \leq C n^{-\frac{2\beta}{2\beta + d}}.
\end{align}
Next, recalling the definition of $M$ and $\beta_{-M}$ as functions of $n$ we note that for any fixed $\beta>0$ it is possible to have $\beta<\beta_{-M}$ only for $n$ not exceeding some finite number $n_0(\beta)$. For such values of $n$ the estimation error of  $\widetilde{f}_n$ is bounded by a constant depending only on $\beta$, $d$ and $L_0$:
\begin{align*}
    \matE \big[\|\widetilde{f}_n - f\|_{L_2}^2\big]  \leq 4 \matE_1 \big[\max_{i=1,\dots,n_0(\beta)/2} Y_i^2\big] + 2L_0^2 \leq C (\log(n_0(\beta))+ L_0^2).
\end{align*}
Consequently, \eqref{eq:a6} also holds for $0<\beta<\beta_{-M}$ (and thus for all $\beta\in(0,\betamax]$) if we take the constant  $C>0$ in \eqref{eq:a6} large enough. 

By the same argument, we deduce that the bound   \eqref{eq:a6} holds for the estimator $\widetilde g_n$. Combining both bounds and using the fact that function $\lambda(\cdot)$ appearing in \eqref{def:final_estimator} takes values in $[0,1]$ we get the desired bound \eqref{eq:thm2} for the final estimator $\hat{\sf f}_n$. 
\end{proof}

\bibliographystyle{apalike}

\end{document}